\documentclass{amsart}
\chardef\bslash=`\\

\usepackage{amssymb,amsmath,amsfonts,amsthm,amscd,stmaryrd,mathrsfs}
\usepackage{hyperref}
\usepackage[all,cmtip,poly]{xy}

\usepackage{silence}
\newtheorem{theorem}[subsection]{Theorem}

\newtheorem{lemma}[subsection]{Lemma}

\newtheorem{introthm}{Theorem}

\newtheorem{cor}[subsection]{Corollary}

\newtheorem{prop}[subsection]{Proposition}
\newtheorem{proposition}[subsection]{Proposition}

\theoremstyle{remark}

\numberwithin{equation}{section}

\newif\iffinalrun
\iffinalrun
\else
 \fi

\iffinalrun
  \newcommand{\need}[1]{}
  \newcommand{\mar}[1]{}
\else
  \newcommand{\need}[1]{{\tiny *** #1}}
  \newcommand{\mar}[1]{\marginpar{\raggedright\tiny  #1}}\fi

\renewcommand\mathbb{\mathbf}

\newcommand{\scrQ}{\mathscr{Q}}
\newcommand{\Lie}{{\operatorname{Lie}\,}}

\newcommand{\rec}{{\operatorname{rec}}}

\newcommand{\wv}{{\widetilde{v}}}

\renewcommand{\ell}{l}

\def\Iw{\mathrm{Iw}}

\newcommand{\St}{\operatorname{St}}
\newcommand{\ad}{\operatorname{ad}}
\newcommand{\diag}{\operatorname{diag}}
\newcommand{\tr}{\operatorname{tr}}

\newcommand{\A}{\mathbf{A}}
\newcommand{\bA}{\ensuremath{\mathbf{A}}}
\newcommand{\CC}{{\mathbb C}}
\newcommand{\bC}{\ensuremath{\mathbf{C}}}

\newcommand{\bQ}{\ensuremath{\mathbf{Q}}}
\newcommand{\Q}{\QQ}
\newcommand{\QQ}{{\mathbb Q}}
\newcommand{\bN}{{\mathbf N}}

\newcommand{\bR}{\ensuremath{\mathbf{R}}}

\newcommand{\bT}{\ensuremath{\mathbf{T}}}

\newcommand{\bZ}{\ensuremath{\mathbf{Z}}}

\newcommand{\bbZ}{\ensuremath{\mathbf{Z}}}

\newcommand{\cA}{{\mathcal A}}
\newcommand{\cB}{{\mathcal B}}
\newcommand{\cC}{{\mathcal C}}

\newcommand{\cF}{{\mathcal F}}
\newcommand{\cG}{{\mathcal G}}
\newcommand{\cH}{{\mathcal H}}

\newcommand{\cL}{{\mathcal L}}

\newcommand{\cO}{{\mathcal O}}
\newcommand{\cR}{{\mathcal R}}
\newcommand{\cS}{{\mathcal S}}

\newcommand{\cV}{{\mathcal V}}

\newcommand{\ffrm}{{\mathfrak m}}

\newcommand{\frakp}{\mathfrak{p}}
\newcommand{\p}{\frakp}
\newcommand{\frakq}{\mathfrak{q}}
\newcommand{\q}{\frakq}

\newcommand{\Qbar}{\overline{\Q}}

\newcommand{\Qpbar}{\Qbar_p}

\DeclareMathOperator{\Ad}{Ad}

\DeclareMathOperator{\End}{End}

\DeclareMathOperator{\Gal}{Gal}
\newcommand{\GL}{\mathrm{GL}}
\newcommand{\GSp}{\mathrm{GSp}}
\DeclareMathOperator{\Hom}{Hom}

\DeclareMathOperator{\ord}{ord}

\DeclareMathOperator{\Spec}{Spec}

\newcommand{\Frob}{\mathrm{Frob}}

\newcommand{\Art}{{\operatorname{Art}}}
\newcommand{\Res}{\operatorname{Res}}

\newcommand{\doubleslash}{/\kern-0.2em{/}}

\newcommand{\prm}{\mathrm{p}}
\newcommand{\mrm}{\mathrm{m}}

\begin{document}
\dedicatory{To Professor Benedict Gross, on the occasion of his 70th birthday}
\author{Jack A. Thorne}
\title[The vanishing of adjoint Selmer groups]{On the vanishing of adjoint Bloch--Kato Selmer groups of irreducible automorphic Galois representations}
\begin{abstract} Let $\rho$ be the $p$-adic Galois representation attached to a cuspidal, regular algebraic, polarizable automorphic representation of $\GL_n$. Assuming only that $\rho$ satisfies an irreducibility condition, we prove the vanishing of the adjoint Bloch--Kato Selmer group attached to $\rho$. This generalizes previous work of the author and James Newton. 
\end{abstract}
\maketitle
\setcounter{tocdepth}{1}
\tableofcontents

\section{Introduction}

\textbf{Context.} Let $F, M$ be number fields, let $\widehat{G}$ be a reductive group defined over $M$, and suppose given a strictly compatible system 
\[ \cR = \left(\rho_\lambda : G_F \to \widehat{G}(\overline{M}_\lambda) \right)_\lambda \]
of continuous, $\widehat{G}$-irreducible $\lambda$-adic Galois representations.\footnote{By `strictly compatible' we mean that for each finite place $v$ of $F$, there is a Weil--Deligne representation $(r_v, N_v)$ of $W_{F_v}$ into $\widehat{G}$ over $\overline{M}$ (all but finitely many of which are unramified) such that for each place $\lambda$ of $M$, the Frobenius-semisimple Weil--Deligne representation associated to $\rho_\lambda|_{W_{F_v}}$ is conjugate to $(r_v, N_v)$. In particular, if $v$ and $\lambda$ have the same residue characteristic then $\rho_\lambda|_{G_{F_v}}$ is de Rham.}  For any place $\lambda$ of $M$ of residue characteristic $\ell$ and any representation $R : \widehat{G} \to \GL_N$, we may define the Bloch--Kato Selmer group of $R \circ \rho_\lambda$: 
\begin{multline*} H^1_f( F, R \circ \rho_\lambda ) \\ = \ker\left( H^1(F, R \circ \rho_\lambda) \to \prod_{v | \ell} H^1(F_v, (R \circ \rho_\lambda) \otimes_{\bQ_\ell} B_{cris}) \times \prod_{v \nmid \ell} H^1_{ur}(F_v, (R \circ \rho_\lambda)) \right).  \end{multline*}
Fixing an embedding $M \to \bC$, we may also define the associated $L$-function 
\[ L(\cR, R, s) = \prod_v \det(1 - (R \circ r_v)^{I_{F_v}, N_v = 0}(\Frob_v) q_v^{-s})^{-1}. \]
Conjectures of Fontaine--Mazur, Beilinson, and Bloch--Kato \cite{Fon95, Bei84, Blo90} together lead to the expectation that $ L(\cR, R, s)$
 converges absolutely in some right half-plane and admits a meromorphic continuation to $\bC$, and moreover that for any $\lambda$ there is an equality
\begin{equation}\label{eqn_intro_BK} \dim H^1_f(F, R \circ \rho_\lambda) - \dim H^0(F, R \circ \rho_\lambda) = \ord_{s = 1} L(\cR, R^\vee, s).  
\end{equation}
We are concerned here with the special case where $R = \Ad_{\widehat{G}}$ is the adjoint representation of $\widehat{G}$, when we should have
\begin{equation}\label{eqn_intro_BK_Ad} \dim H^1_f(F, \Ad_{\widehat{G}} \rho_\lambda) - \dim H^0(F,\Ad_{\widehat{G}}  \rho_\lambda) = \ord_{s = 1} L(\cR, \Ad_{\widehat{G}}, s).  
\end{equation}
 The representations $\Ad_{\widehat{G}} \rho_\lambda$ should be pure of weight 0, and one expects the group $\dim H^1_f(F, \Ad_{\widehat{G}} \rho_\lambda)$ to vanish. Since the representation $\Ad_{\widehat{G}}$ is self-dual, (\ref{eqn_intro_BK_Ad}) is expected to be equivalent (applying Poitou--Tate duality to the left-hand side and functional equation of the $L$-function to the right-hand side) to the equality
\begin{equation}\label{eqn_intro_BK_Ad_dual} \dim H^1_f(F, \Ad_{\widehat{G}} \rho_\lambda(1)) - \dim H^0(F,\Ad_{\widehat{G}} \rho_\lambda(1)) = \ord_{s = 0} L(\cR, \Ad_{\widehat{G}}, s).  
\end{equation}
An interesting case arises when the number field $F$ is totally real, $\widehat{G}$ is the $L$-group of a reductive group $G$ over $F$, and $\rho_\lambda$ is the compatible system of Galois representations conjecturally attached by Buzzard--Gee \cite{Buz14} to an automorphic representation $\pi$ of $G(\A_F)$ such that $\pi_\infty$ is square-integrable. Gross predicted \cite{grossodd} that the representations $\rho_\lambda$ should then be odd, in the sense that for each place $v | \infty$ and complex conjugation $c_v \in G_F$, $\Ad \rho_\lambda(c_v)$ is the unique (up to conjugacy) involution of $\widehat{G}$ such that the trace on $\widehat{\mathfrak{g}}$ equals $- \operatorname{rank} \widehat{G}$. Poitou--Tate duality then implies the equality
\begin{equation}\label{eqn_num_coinc} \dim H^1_f(F, \Ad_{\widehat{G}} \rho_\lambda) = \dim H^1_f(F, \Ad_{\widehat{G}} \rho_\lambda(1)). 
\end{equation}
In this paper we essentially establish the equalities (\ref{eqn_intro_BK_Ad}) and (\ref{eqn_intro_BK_Ad_dual}) for many compatible systems associated to automorphic representations $\pi$ of classical groups $G$ over totally real fields $F$ such that $\pi_\infty$ is discrete series. We are able to do this because the equality (\ref{eqn_num_coinc})
is exactly the `numerical coincidence', described in the introduction to \cite{cht}, under which the Taylor--Wiles method applies. Using the Taylor--Wiles method, we can identify the Bloch--Kato Selmer group of $\Ad_{\widehat{G}} \rho_\lambda$ with the Zariski tangent space of a Hecke algebra acting on a space of cuspidal automorphic forms. The vanishing of the Selmer group is thus ultimately a consequence of the fact that this action is semisimple.

This theme, sometimes with integral refinements, has been explored by several authors (see e.g.\ \cite{Kis04a, Dia04, All16, New19a}). On the other hand, Calegari--Geraghty \cite{Cal18} have recently explained how the Taylor--Wiles method can be generalized to cases where the numerical coincidence no longer holds, and applied this, with Harris, to prove unconditionally the vanishing of the adjoint Bloch--Kato Selmer group in some cases for automorphic representations of $\GSp_4(\bA_\bQ)$ associated to abelian surfaces over $\bQ$ \cite[Theorem A.1]{Cal20}. Our aim here is to leverage the relative maturity of the Taylor--Wiles case to prove vanishing results that are as general as possible. 

\textbf{Results.} To state our results, we prefer to work with automorphic representations of general linear groups satisfying self-duality conditions. Let $F$ be a CM number field, and let $\pi$ be a cuspidal, regular algebraic automorphic representation of $\GL_n(\A_F)$ which is polarizable, in the sense of \cite{BLGGT}. Then for any prime $p$ and isomorphism $\iota : \overline{\bQ}_p \to \bC$, there is an associated Galois representation $r_{\pi, \iota} : G_F \to \GL_n(\overline{\bQ}_p)$. Since $\pi$ is polarizable, $r_{\pi, \iota}$ is conjugate self-dual up to twist, and $\Ad  r_{\pi, \iota}$ extends to a representation of $G_{F^+}$ on $M_n(\overline{\bQ}_p)$ (which we may think of arising from the adjoint representation of the $L$-group of a unitary group over $F^+$). This defines the associated adjoint Bloch--Kato Selmer group $H^1_f(F^+, \Ad r_{\pi, \iota})$.

In a previous paper \cite{New19a}, we proved that this adjoint Selmer group vanishes provided that the group $r_{\pi, \iota}(G_{F(\zeta_{p^\infty})})$ is ``enormous''; roughly speaking, that it contains enough regular semisimple elements. The main theorem of this paper strengthens this result, proving the same vanishing under the weaker condition that $r_{\pi, \iota}|_{G_{F(\zeta_{p^\infty})}}$ is irreducible:
\begin{introthm}\label{thm_intro_main_theorem}[Theorem \ref{thm_vanishing_over_CM_field}]
	Let $F$ be a CM number field, and let $\pi$ be a polarizable, cuspidal, regular algebraic automorphic representation of $\GL_n(\A_F)$. Let $p$ be a prime, and let $\iota : \overline{\bQ}_p \to \bC$ be an isomorphism. Suppose that $r_{\pi, \iota}|_{G_{F(\zeta_{p^\infty})}}$ is irreducible. Then $H^1_f(F^+, \ad r_{\pi, \iota}) = 0$.
\end{introthm}
This theorem is probably the best possible using the kinds of methods considered here. We hope that this theorem will have applications of a similar sort to those of the main result of \cite{New19a} (see for example the papers \cite{New19b, New20}). For an analogous theorem in the case where $F$ is a totally real field, see Theorem \ref{thm_vanishing_over_real_field} below.

We now explain what is new here compared to the arguments of \cite{New19a}. As in that paper, we show that $H^1_f(F^+, \ad r_{\pi, \iota}) = 0$ by using auxiliary Selmer groups, with torsion coefficients, and where we allow ramification at Taylor--Wiles places of the base number field. Previously, we considered ramification at places where the image of Frobenius under $r_{\pi, \iota}$ is regular semisimple, with the modified Selmer conditions allowing arbitrary ramification at these places. Here we do not impose any condition on the image of Frobenius. However, we must then cut down the relevant Selmer condition, as allowing arbitrary ramification would otherwise define a Selmer group that was `too large'. The condition we impose is roughly that, selecting an eigenvalue $\alpha$ of the Frobenius at a Taylor--Wiles place, inertia acts through a scalar character on the $\alpha$-generalized eigenspace (an idea similar to the one used in \cite{jack}).

The hardest part of the proof is showing that this condition makes sense both at the level of Galois deformation theory and at the level of automorphic forms. We note that as in \cite{New19a}, we impose no condition on the residual representation $\overline{r}_{\pi, \iota}$ (which might even be trivial), so we need to study carefully the interaction of these conditions with the various integral structures that appear in order to make the final patching argument go through.

\textbf{Organization of this paper.} In \S \ref{sec_different}, we compute the different of the ring extension $\bbZ[x_1, \dots, x_n]^{S_n} \to \bbZ[x_1, \dots,  x_n]^{S_{n_1} \times S_{n_2}}$: it is the resultant, and occurs constantly throughout this paper. In \S \ref{sec_hecke}, we realise this ring extension as a map of Hecke algebras and show how the different controls the difference between certain spaces of automorphic forms which naturally appear in the Taylor--Wiles method. In \S \ref{sec_adequacy} and \S \ref{sec_deformation} we study our auxiliary Selmer groups. Finally, in \S \ref{sec_patching}, we combine everything to prove Theorem \ref{thm_intro_main_theorem}.

\textbf{Acknowledgements.} This work received funding from the European 
Research Council (ERC) under the European Union's Horizon 2020 research and 
innovation programme (grant agreement No 714405). 

\textbf{Notation.} We use the same notation as defined in \cite[\S 1]{New19a}. Table \ref{tab} gives a list of symbols used with their meanings. We refer to \emph{loc. cit.} for precise definitions.
\begin{table}
\caption{Summary of notation}\label{tab}
\begin{tabular}{|p{0.2\textwidth}|p{0.7\textwidth}|}
\hline
 Symbol &  Meaning \\ \hline
 $G_F$ & Absolute Galois group of field $F$ of characteristic 0 \\ \hline
$F_v$, $\cO_{F_v}$, $\varpi_v$, $k(v)$ &  Completion of number field $F$ at finite place $v$, ring of integers, fixed choice of uniformizer, residue field \\  \hline
 $F_S / F$, $G_{F, S}$& Maximal extension of number field $F$ unramified outside finite set $S$, $\Gal(F_S / F)$ \\ \hline
 $E, \cO, \varpi, k$ &  Finite extension of $\bQ_p$ with ring of integers, uniformizer, residue field \\ \hline
 $\cC_\cO, \cC_E$ & Category of complete Noetherian local $\cO$, resp. $E$-algebras \\ \hline
 $\cH(G, U)$ & Hecke algebra of locally profinite group $G$ with identity element $[U]$ \\ \hline
 $W_K, I_K, \Art_K$ & Weil group, inertia group, Artin map of $p$-adic local field $K$ \\ \hline
 $\rec_K, \rec^T_K$ & Local Langlands correspondence for $\GL_n(K)$ and its Tate-normalised version  \\ \hline
 $\mathrm{WD}(\rho)$, $\mathrm{WD}(\rho)^{F-ss}$ & Weil--Deligne representation associated to continuous representation $\rho : G_K \to \GL_n(\overline{\bQ}_\ell)$ (assumed geometric if $\ell = p$) and its Frobenius-semisimplification \\ \hline
 $r_{\pi, \iota}$ & $p$-adic Galois representation associated to a regular algebraic, cuspidal, polarizable automorphic representation of $\GL_n(\A_F)$, $F$ a CM or totally real number field, and $\iota : \overline{\bQ}_p \to \bC$ an isomorphism \\ \hline
 $\cG_n, \ad $ & Group scheme with neutral component $\GL_n \times \GL_1$ considered in \cite[\S 2]{cht}, and its adjoint representation on $\Lie \GL_n$ \\ \hline
\end{tabular}
	\end{table}
	\section{A different computation}\label{sec_different}

Let $\cA = \bbZ[e_1, \dots, e_n] \subset \cC = \bbZ[x_1, \dots, x_n]$ denote the ring of symmetric polynomials in $n$ variables. Fix a decomposition $n = n_1 + n_2$, and define elements $a_1, \dots, a_{n_1}$ and $b_1, \dots, b_{n_2}$ by the relations
\[ (T - x_1) \dots (T - x_{n_1}) = \sum_{i=0}^{n_1} T^{n_1-i}  a_{i}, \]
\[ (T - x_{n_1 + 1}) \dots (T - x_{n}) = \sum_{i=0}^{n_2} T^{n_2-i} b_{i}, \]
hence
\[ (T - x_1) \dots (T - x_n) = \sum_{i=0}^n T^{n-i}  e_i = \left(\sum_{i=0}^{n_1} T^{n_1-i}  a_{i}\right)\left(\sum_{j=0}^{n_2} T^{n_2-j} b_{j}\right). \]
We define $\cB = \cB_{n_1, n_2} = \bbZ[a_1, \dots, a_{n_1}, b_1, \dots, b_{n_2}]$. Thus the group $S_n$ acts on $\cC$ and its subgroup $S_{n_1} \times S_{n_2}$ acts trivially on $\cB$. 
\begin{lemma}
The ring $\cB$ is the ring of polynomials invariant under the subgroup $S_{n_1} \times S_{n_2}$, and is a free $\cA$-module. 
\end{lemma}
\begin{proof}
Let $R$ be a commutative ring, and let $S$ be an $R$-algebra. Then \cite[Ch. IV, \S 6, No. 1, Theorem 1]{Bou03} implies that the canonical map $R[x_1, \dots, x_n]^{S_n} \otimes_R S \to S[x_1, \dots, x_n]^{S_n}$ is an isomorphism (as both source and target are free $S$-modules, with basis given by the monomials in the elementary symmetric polynomials) and that $R[x_1, \dots, x_n]$ is a free $R[x_1, \dots, x_n]^{S_n}$-module. 

Let $R_1 = \bZ[x_1, \dots, x_{n_1}]$ and $R_2 = \bZ[x_{n - n_2 + 1}, \dots, x_n]$. We deduce that
\[ \cB =  R_1^{S_{n_1}} \otimes_\bZ R_2^{S_{n_2}} \cong (R_1^{S_{n_1}} \otimes_\bZ R_2)^{S_{n_2}} \cong (R_1 \otimes_\bZ R_2)^{S_{n_1} \times S_{n_2}}. \]
This proves the first statement of the lemma. We also see that $\cC = R_1 \otimes_\bZ R_2$ is free both as $\cA$-module and as $\cB$-module, and so there is an isomorphism $\cB^{ n_1! n_2!} \cong \cA^{n!}$ of graded $\cA$-modules. In particular, $\cB$ is a finitely generated projective graded $\cA$-module, which must therefore be free. 
\end{proof}
An important element of $\cB$ is the resultant 
\[ \Res_{n_1, n_2} = \operatorname{Res}(\sum_{i=0}^{n_1} T^{n_1-i}  a_{i},   \sum_{i=0}^{n_2} T^{n_2-i}  b_{i}) = \prod_{i=1}^{n_1} \prod_{j=1}^{n_2} (x_i - x_{n_1+j}).\]
Another important element of $\cB$ is the $q$-resultant, defined for $q \in \bN$:
\[ \Res_{q, n_1, n_2} =  \operatorname{Res}(\sum_{i=0}^{n_1} T^{n_1-i} q^i a_{i},   \sum_{i=0}^{n_2} T^{n_2-i}  b_{i}) = \prod_{i=1}^{n_1} \prod_{j=1}^{n_2} (q x_i - x_{n_1+j}).\]
\begin{proposition}\label{prop_different}
There exists a unique element $\widetilde{\Res}_{n_1, n_2} = \sum_i z_i \otimes w_i \in \cB \otimes_\cA \cB$ with the following properties:
\begin{enumerate}
\item $\sum_i z_i w_i = {\Res_{n_1, n_2}}$.
\item For each $\sigma \in S_n - S_{n_1} \times S_{n_2}$, $\sum_i \sigma(z_i) w_i = 0$.
\end{enumerate}
\end{proposition}
\begin{proof}
Let $\mu : \cB \otimes_\cA \cB \to \cB$ be the $\cA$-algebra homomorphism given by $\mu(z \otimes w) = zw$. Let $I = \ker(\mu)$, and let $J = \operatorname{Ann}_{\cB \otimes_\cA \cB}(I)$. The statement of the proposition is equivalent to the assertion that the map $\mu|_J$ is injective and its image contains ${\Res_{n_1, n_2}}$. In fact, we will show that $\mu|_J$ is an isomorphism onto the ideal of $\cB$ generated by ${\Res_{n_1, n_2}}$.

The ring extension $\cA \subset \cB$ satisfies the hypotheses of \cite[Lemma 0BWD]{stacks-project}, which implies that $\mu|_J$ is an isomorphism onto the ideal of $\cB$ generated by the determinant of the Jacobian matrix $(\partial (e_1, \dots, e_n) / \partial (a_1, \dots, a_{n_1}, b_1, \dots, b_{n_2}))$. The determinant of this matrix is (up to sign) ${\Res_{n_1, n_2}}$.
\end{proof}
\begin{proposition}\label{prop_etale_when_different_non-zero}
The morphism $\Spec \cB \to \Spec \cA$ is \'etale away from the locus ${\Res_{n_1, n_2}} = 0$.
\end{proposition}
\begin{proof}
The proof of Proposition \ref{prop_different} shows that ${\Res_{n_1, n_2}}$ generates the Noether different of $\cA \to \cB$. The morphism $\cA \to \cB$ is flat, and \cite[\href{https://stacks.math.columbia.edu/tag/0BVU}{Tag 0BVU}]{stacks-project} shows that $\Spec \cB \to \Spec \cA$ fails to be unramified precisely at the points defined by the equation ${\Res_{n_1, n_2}} = 0$. 
\end{proof}
In this paper we will frequently use the interpretation of $\Spec \cB$ as the scheme of factorisations $F(X) = F_1(X) F_2(X)$, where $F_1, F_2$ are monic of degrees $n_1, n_2$, respectively. A related construction is given by the following lemma.
\begin{lemma}\label{lem_universal_idempotents} There are unique polynomials $G_1(X), G_2(X) \in \cB[X]$ of degrees $n_2-1, n_1-1$ such that $G_1(X) F_1(X) + G_2(X) F_2(X) = {\Res_{n_1, n_2}}$.
\end{lemma}
\begin{proof}
	For a ring $R$, let $\operatorname{Pol}_d(R)$ denote the free $R$-module of polynomials of degree $\leq d$ with coefficients in $R$. There is a morphism $\mu : \operatorname{Pol}_{n_2 - 1}(\cB) \times \operatorname{Pol}_{n_1 - 1}(\cB) \to \operatorname{Pol}_{n_1+n_2-1}(\cB)$, $(G_1, G_2) \mapsto G_1 F_1 + G_2 F_2$. With respect to the standard bases the matrix of this morphism is the Sylvester matrix, whose determinant is the resultant ${\Res_{n_1, n_2}}$. The existence follows from the existence of the adjugate matrix and the uniqueness from linear algebra over $\operatorname{Frac} \cB$.
\end{proof}
Motivated by the lemma, we define 
\begin{equation}\label{eqn_projectors}
e_1(X) = G_2(X) F_2(X) \text{ and }e_2(X) = G_1(X) F_1(X),
\end{equation}
so that $e_1(X), e_2(X) \in \cB[X]$ and $e_1(X) + e_2(X) = {\Res_{n_1, n_2}}$.

In the statement of the next lemma, we fix a discrete valuation ring $\cO$ with uniformizer $\varpi$, and for each $m \geq 1$ define $A_m = \cO \oplus \epsilon \cO / \varpi^m \cO$ (where $\epsilon^2 = 0$) and, if $x \in \cO$, write $\alpha_x : A_m \to A_m$ for the $\cO$-algebra homomorphism which sends $\epsilon$ to $\epsilon x$.
\begin{lemma}\label{lem_approximate_hensel}
Let $f(X) \in \cO[X]$ be a monic polynomial of degree $n \geq 1$, and suppose given a factorisation $f(X) = f_1(X) f_2(X)$ in $\cO[X]$, where $f_1(X), f_2(X)$ are monic polynomials of degrees $n_1, n_2$, respectively. Suppose that the resultant $\delta = \operatorname{Res}(f_1, f_2) \in \cO$ is non-zero. Then:
\begin{enumerate}
\item There exist unique polynomials $g_1(X), g_2(X) \in \cO[X]$ of degrees strictly less than $n_2, n_1$, respectively, such that $g_1(X) f_1(X) + g_2(X) f_2(X) = \delta$.
\item Let $m \geq 1$, and suppose given a monic polynomial $\widetilde{f}(X) \in A_m[X]$ such that $\widetilde{f}(X) \text{ mod } \epsilon = f(X)$. Then there exists a factorisation $\alpha_\delta(\widetilde{f}(X)) = \widetilde{f}_1(X) \widetilde{f}_2(X)$ in $A_m[X]$, where $\widetilde{f}_1(X)$, $\widetilde{f}_2(X)$ are monic polynomials such that $\widetilde{f}_i(X) \text{ mod } \epsilon = f_i(X)$ ($i = 1, 2$).
\end{enumerate}
\end{lemma}
\begin{proof}
The proof of the first part is essentially the same as the proof of Lemma \ref{lem_universal_idempotents} (except we replace $\cB$ by $\cO$).
For the second, write $\widetilde{f}(X) = f(X) + \epsilon h(X)$, where $h(X) \in P_{n-1}(\cO / \varpi^m)$. If $\widetilde{f}_i(X) = f_i(X) + \epsilon h_i(X)$, then $\widetilde{f}_1(X) \widetilde{f}_2(X) = f(X) + \epsilon( h_2(X) f_1(X) + h_1(X) f_2(X))$. Solving $\alpha_\delta(\widetilde{f}(X)) = \widetilde{f}_1(X) \widetilde{f}_2(X)$ is therefore equivalent to solving $\delta h(X) =h_2(X) f_1(X) + h_1(X) f_2(X)$, which we can do by choosing $(h_2, h_1)$ to be the image of $h(X)$ under the adjugate of the morphism $\mu$ considered in the proof of Lemma \ref{lem_universal_idempotents}.
\end{proof}

\section{Parahoric Hecke algebras}\label{sec_hecke}

Let $v$ be a finite place of a number field $F$, and let $G$ be a split reductive group over $\cO_{F_v}$ (we will soon specialise  to the case $G = \GL_n$). Fix a choice of split maximal torus and Borel subgroup $T \subset B \subset G$. If $P \subset G$ is a standard parabolic subgroup, then we let $P = M_P  N_P$ denote the standard Levi decomposition of $P$. Define the modulus character $\delta_P : P(F_v) \to \bbZ[q_v^{\pm 1/2}]^\times$ by $\delta_P(p) = | \det(\Ad(p)|_{\Lie N_P}) |_v$.  Let $W_G= W(G, T)$ denote the Weyl group of $G$, and $W_{M_P} = W(M_P, T)$ the Weyl group of $M_P$. Thus $W_G$ acts  on  $X_\ast(T)$  on the  left. 

If $A$ is a $\bbZ[q_v^{\pm 1/2}]$-algebra, and if $\pi$ is a smooth $A[G(F_v)]$-module, then we write $\pi_{N_P}$ for the space of $N_P(F_v)$-coinvariants of $\pi$, considered as $A[M_P(F_v)]$-module, and $r_P(\pi) = \pi_{N_P}(\delta_P^{-1/2})$. Thus $r_P(\pi)$ is what we usually call the normalised Jacquet module of $\pi$. If $A = \bC$ and $\pi = i_{B(F_v)}^{G(F_v)} \chi$ is an unramified principal series representation (i.e. the normalised induction of the inflation of $\chi$ to a character of $B(F_v)$)), then the characters of $T(F_v)$ appearing in $r_B(\pi)$ are those in the Weyl orbit of $\chi$.

Let $\mathfrak{p} \subset G(\cO_{F_v})$ be the standard parahoric subgroup associated to $P$ (pre-image of $P(k(v))$ in $P(\cO_{F_v})$). It contains the standard Iwahori subgroup $\mathfrak{b}$. If $P \subset Q$ are standard parabolic subgroups of $G$ then there is a natural inclusion $\cH(G(F_v), \q) \subset \cH(G(F_v), \p)$ which is not an algebra homomorphism, since it does not preserve unit elements. We will write $\cdot_\q$ for the multiplication in $\cH(G(F_v), \q)$ (so e.g. $[\q] \cdot_\p [\q] = [ \q : \p] [\q]$). Similarly if $\pi$ is a $\bbZ[q_v^{\pm 1/2}]$-module then we endow $\pi^\p$ with its natural structure of $\cH(G(F_v), \p)$-module; if $t \in \cH(G(F_v), \p)$ and $x \in \pi^\p$, then we write multiplication as $t \cdot_\p x$, so $[\p] \cdot_\p x = x$ for all $x \in \pi^\p$.

Given a standard parabolic subgroup $P$, let $\mathfrak{m}_P = M_P(\cO_{F_v}) = \mathfrak{p} \cap M_P(F_v)$. Let $\overline{N}_P$ denote the unipotent radical of the opposite parabolic with Levi subgroup $M_P$. Define $\mathfrak{n}_P = N_P(\cO_{F_v}) = \mathfrak{p} \cap N_P(F_v)$ and $\overline{\mathfrak{n}}_P = \ker(\overline{N}_P(\cO_{F_v}) \to \overline{N}_P(k(v))) = \mathfrak{p} \cap \overline{N}_P(F_v)$. Then $\mathfrak{p}$ has its Iwasawa decomposition
\[ \mathfrak{p} = \overline{ \mathfrak{n}}_P \mathfrak{m} \mathfrak{n}_P. \]
We say that an element $m \in M_P(F_v)$ is positive if we have the inclusions
\[ m\mathfrak{n}_P m^{-1} \subset \mathfrak{n}_P \] and 
\[ m^{-1} \overline{ \mathfrak{n}}_P  m \subset \overline{ \mathfrak{n}}_P . \]
We write $M_P(F_v)^+ \subset M_P(F_v)$ for the submonoid of positive elements. 
\begin{lemma}\label{lem_inverse_to_Hecke}
Let $\lambda \in X_\ast(T)$ be a dominant cocharacter which is valued in the centre $Z(M_P)$ of $M_P$. Then $[\mathfrak{p} \lambda(\varpi_v) \mathfrak{p}]$ is an invertible element of $\cH(G(F_v), \mathfrak{p}) \otimes_\bbZ \bbZ[q_v^{\pm 1/2}]$.
\end{lemma}
\begin{proof}
We first recall the Iwahori--Matsumoto presentation of the Iwahori--Hecke algebra $\cH(G(F_v), \mathfrak{b}) \otimes_\bZ \bZ[q_v^{\pm 1/2}]$. Define the affine Weyl group 
\[ \widetilde{W}_G = N_{G(F_v)}(T) / T(\cO_{F_v}) \cong W_G \ltimes X_\ast(T). \] The choice of $\mathfrak{b}$ determines a set of simple affine roots, hence a set of simple affine reflections in $\widetilde{W}_G$; they are the linear reflections $s_\alpha \in W_G$ associated to the simple roots of the pair $(B, T)$, together with the affine reflections $s_{\alpha_0} \alpha_0^\vee$, where $\alpha_0$ is the lowest root of a simple sub-root system of $\Phi(G, T)$. There is an associated length function $l : \widetilde{W}_G \to \bZ_{\geq 0}$ giving each simple affine reflection length 1. The associated braid group $B_G$ is the free group generated by the elements $T_w$ ($w \in \widetilde{W}_G$) subject to the relations $T_{w w'} = T_w T_{w'}$ when $l(w w') = l(w) l(w')$. The Iwahori--Matsumoto presentation is a surjective algebra homomorphism
\[  \bZ[q_v^{\pm 1/2}][B_G] \to  \cH(G(F_v), \mathfrak{b}) \otimes_\bZ \bZ[q_v^{\pm 1/2}], \]
\[ T_w \mapsto [ \mathfrak{b} w \mathfrak{b}], \]
with kernel generated by the elements $(T_s + 1)(T_s - q_v)$, $s$ a simple affine reflection.

If $w \in W_{M_P}$ and $\lambda \in X_\ast(T)$ is a cocharacter valued in $Z(M_P)$, then $w \lambda = \lambda$, which implies (\cite[Lemma 2.2]{Lus89}) that $T_w$ and $T_\lambda$ commute in $B_G$. We deduce that $[\mathfrak{p}]  = \sum_{w \in W_{M_P}} [\mathfrak{b} w \mathfrak{b}]$ and $[\mathfrak{b} \lambda(\varpi_v) \mathfrak{b}]$ commute in $\cH(G(F_v), \mathfrak{b})$. We also see that $[\mathfrak{b} \lambda(\varpi_v) \mathfrak{b}]$ is an invertible element of $\cH(G(F_v), \mathfrak{b}) \otimes_\bZ \bZ[q_v^{\pm 1/2}]$; let $t \in \cH(G(F_v), \mathfrak{b}) \otimes_\bZ \bZ[q_v^{\pm 1/2}]$ be its inverse.

The element $t$ commutes with $[\mathfrak{p}]$, while direct computation shows that $[\mathfrak{p}] \cdot_\mathfrak{b} [\mathfrak{b} \lambda(\varpi_v) \mathfrak{b}] = [\mathfrak{p} \lambda(\varpi_v) \mathfrak{p}]$. Let $t' = t \cdot_{\mathfrak{b}} [\mathfrak{p}] \in \cH(G(F_v), \mathfrak{p}) \otimes_\bbZ \bbZ[q_v^{\pm 1/2}]$. We finally compute
\[ t' \cdot_{\mathfrak{p}} [\mathfrak{p} \lambda(\varpi_v) \mathfrak{p}] = (t \cdot_{\mathfrak{b}} [\mathfrak{p}]) \cdot_{\mathfrak{p}} ( [\mathfrak{b} \lambda(\varpi_v) \mathfrak{b}] \cdot_{\mathfrak{b}} [\mathfrak{p}] )= [\mathfrak{p}], \]
and similarly $[\mathfrak{p} \lambda(\varpi_v) \mathfrak{p}]\cdot_{\mathfrak{p}} t' = [\mathfrak{p}]$. This completes the proof.
\end{proof}
The following proposition is basically contained in \cite{Bus98} and \cite{Vig98}. 
\begin{proposition}\label{prop_jacquet_module}
	\begin{enumerate}
		\item $\cH(M_P(F_v)^+, \mathfrak{m}_P) \subset \cH(M_P(F_v), \mathfrak{m}_P)$ is a subalgebra.
		\item The homomorphism of $\bbZ[q_v^{\pm 1/2}]$-modules 
		\[ T_P^+ : \cH(M_P(F_v)^+, \mathfrak{m}_P) \otimes_\bbZ \bbZ[q_v^{\pm 1/2}] \to \cH(G(F_v), \mathfrak{p}) \otimes_\bbZ \bbZ[q_v^{\pm 1/2}] \]
		 defined on basis elements by 
		\[ T_P^+([\mathfrak{m}_P m \mathfrak{m}_P]) = \delta_P(m)^{1/2} [\mathfrak{p} m \mathfrak{p}] \]
		 is an injective algebra homomorphism, with image equal to the set of functions with support in $\mathfrak{p} M_P(F_v)^+ \mathfrak{p}$. 
		\item $T_P^+$ extends uniquely to an algebra homomorphism $T_P : \cH(M_P(F_v), \mathfrak{m}_P) \otimes_\bZ \bZ[q_v^{\pm 1/2}] \to \cH( G(F_v), \mathfrak{p})\otimes_\bZ \bZ[q_v^{\pm 1/2}] $.
		\item Let $\pi$ be a smooth $\bZ[q_v^{\pm 1/2}] [G(F_v)]$-module, and let $q : \pi^{\mathfrak{p}} \to r_P(\pi)^{\mathfrak{m}_P}$ denote the canonical projection. Then $q$ is an isomorphism and for any $x \in \pi^{\mathfrak{p}}$, $t \in \cH(M_P(F_v), \mathfrak{m}_P) \otimes_\bZ \bZ[q_v^{\pm 1/2}]$, we have $q(T_P(t) x) = t q(x)$.
	\end{enumerate}
\end{proposition}
\begin{proof}
The first two parts follow from \cite[Corollary  6.12]{Bus98}. The third part follows from \cite[Theorem 7.2]{Bus98}, provided we can exhibit an element $z \in Z({M_P})(F_v)$ such that $[\mathfrak{p} z \mathfrak{p}]$ is invertible  in $\cH( G(F_v), \mathfrak{p})\otimes_\bZ \bZ[q_v^{\pm 1/2}] $ and such that $z$ is strongly $(P, \mathfrak{p})$-positive, in the sense of \cite[Definition 6.16]{Bus98}. The existence of such an element follows from Lemma \ref{lem_inverse_to_Hecke}.

For the fourth part, we first prove the formula $q(T_P(t)x) = t q(x)$. Because of the presence of invertible elements, it is enough to show that the formula holds for elements of the form $t = [\mathfrak{m}_P m \mathfrak{m}_P]$ with $m \in M_P(F_v)^+$. Choose elements $x_i \in M_P(F_v)$ such that $\mathfrak{m}_P m \mathfrak{m}_P = \sqcup_i x_i \mathfrak{m}_P$, and elements $y_{ij}$ such that $\mathfrak{n}_P = \sqcup_{i, j} y_{ij} x_i \mathfrak{n}_P x_i^{-1}$. Then the number of $y_{ij}$ is $\delta_P(x_i)^{-1} = \delta_P(m)^{-1}$ and $\mathfrak{p} m \mathfrak{p} = \sqcup_{i,j} y_{ij} x_i \mathfrak{p}$ (here we use the positivity of $m$). We then compute
\[ q(T_P(t)x) = \delta_P^{1/2}(m)\sum_{i, j} q(y_{ij} x_i x) =  \delta_P^{-1/2}(m) \sum_i  q(x_i x) = \sum_i x_i q(x) = tq(x). \]
We next show that  $q$ is injective. If $x \in \pi^{\mathfrak{p}}$ and $q(x) = 0$ then we can write $x = \sum_i (n_i - 1) x_i$ for some $n_i \in N(F_v)$, $x_i \in  \pi$. Let $\mathfrak{n}_0 \subset N(F_v)$ be a compact open subgroup containing all of the $n_i$ and $\mathfrak{n}_P$. Then we have $\tr_{\mathfrak{n}_P / \mathfrak{n}_0}(x) = 0$. Choose $m \geq 0$ such that $z^m \mathfrak{n}_0 z^{-m} \subset \mathfrak{n}_P$. Then we have
\[ [\mathfrak{p} z^m \mathfrak{p}] x = \tr_{z^m \mathfrak{n}_P z^{-m} /\mathfrak{n}_P } z^m x = z^m \tr_{ \mathfrak{n}_P / z^{-m} \mathfrak{n}_P z^m} x = 0. \]
Since $ [\mathfrak{p} z^m \mathfrak{p}] $ acts invertibly on $\pi^{\mathfrak{p}}$, we find that $x = 0$.

We finally show that $q$ is surjective. Let $\bar{x} \in r_P(\pi)^{\mathfrak{m}_P}$, and let $x \in \pi$ be a pre-image. We can assume that $x$ is fixed by $\mathfrak{n}_P$. We claim that we can further choose $x$ to be invariant under $\mathfrak{m}_P$. Indeed, for any $g \in \mathfrak{m}_P$, $gx - x$ maps to zero in $r_P(\pi)$. Using the argument of the previous paragraph, we can find an open compact subgroup $\mathfrak{n}_0 \subset N_P(F_v)$ containing $\mathfrak{n}_P$, normalized by $\mathfrak{m}_P$, such that $\tr_{\mathfrak{n}_P / \mathfrak{n}_0} (gx - x) = 0$ for all $g \in \mathfrak{m}_P$. Since $g$ normalises both $\mathfrak{n}_P$ and $\mathfrak{n}_0$, this implies that $g [\mathfrak{n}_0 : \mathfrak{n}_P]^{-1} \tr_{\mathfrak{n}_P / \mathfrak{n}_0} x = [\mathfrak{n}_0 : \mathfrak{n}_P]^{-1} \tr_{\mathfrak{n}_P / \mathfrak{n}_0} x$. We see that $ [\mathfrak{n}_0 : \mathfrak{n}_P]^{-1} \tr_{\mathfrak{n}_P / \mathfrak{n}_0} x $ is a pre-image of $\bar{x}$ which is invariant under $\mathfrak{m}_P \mathfrak{n}_P$.

We can find $m \geq 0$ such that $z^m x$ is invariant under $\overline{\mathfrak{n}}_P$, hence under $\overline{\mathfrak{n}}_P \mathfrak{m}_P z^{m} \mathfrak{n}_P z^{-m}$. It follows that $\tr_{ z^{m} \mathfrak{n}_P z^{-m} / \mathfrak{n}_P} z^m x$ is invariant under $\p$. We finally find that 
\[ [\mathfrak{p} z^m \mathfrak{p}]^{-1} \tr_{z^{m} \mathfrak{n}_P z^{-m} / \mathfrak{n}_P} z^m x \]
 lies in $\pi^{\mathfrak{p}}$ and is the desired pre-image of $x$.
\end{proof}
The Satake isomorphism is a canonical isomorphism (see \cite{Gro96})
\[ \cS_{M_P} :  \cH(M_P(F_v), \mathfrak{m}_P) \otimes_{\bbZ} \bbZ[q_v^{\pm 1/2}] \to (\bbZ[X_\ast(T)] \otimes_{\bbZ} \bbZ[q_v^{\pm 1/2}])^{W_{M_P}} \]
\[	f \mapsto (\cS_{M_P} f)(t) = \delta_{M_P \cap B}(t)^{1/2} \int_{n \in (N_B \cap M_P)(F_v)} f(t n) \, dn.
\]
We define 
\[ \Sigma_P = T_P \circ \cS_{M_P}^{-1} : (\bbZ[X_\ast(T)] \otimes_{\bbZ} \bbZ[q_v^{\pm 1/2}])^{W_{M_P}} \to \cH( G(F_v), \mathfrak{p})  \otimes_{\bbZ} \bbZ[q_v^{\pm 1/2}]. \]
If $\pi$ is any smooth $\bZ[q_v^{\pm 1/2}] [G(F_v)]$-module, then we regard  $\pi^{\mathfrak{p}}$ as a $ (\bbZ[X_\ast(T)] \otimes_{\bbZ} \bbZ[q_v^{\pm 1/2}])^{W_{M_P}}$-module via the map $\Sigma_P$.
\begin{cor}\label{cor_compatible_with_jacquet_module}
Let $\pi$ be a smooth $\bZ[q_v^{\pm 1/2}] [G(F_v)]$-module, and let $q : \pi^{\mathfrak{p}} \to r_P(\pi)^{\mathfrak{m}_P}$ denote the canonical projection. Then $q$ is an isomorphism and for any $s \in (\bbZ[X_\ast(T)] \otimes_{\bbZ} \bbZ[q_v^{\pm 1/2}])^{W_{M_P}}$ and $x \in \pi^{\mathfrak{p}}$,  we have $q(\Sigma_P(s)x) = \cS_{M_P}^{-1}(s) q(x)$.

More generally, let $P \subset Q$ be another standard parabolic subgroup and let $q : \pi^{\mathfrak{p}} \to r_Q(\pi)^{\mathfrak{p} \cap \mathfrak{m}_Q}$ denote the canonical projection. Then $q$ is an isomorphism and for any $s \in (\bbZ[X_\ast(T)] \otimes_{\bbZ} \bbZ[q_v^{\pm 1/2}])^{W_{M_P}}$ and $x \in \pi^{\mathfrak{p}}$, we have $q(\Sigma_P(s) x) = \Sigma_{P \cap M_Q}(s)q(x)$.
\end{cor}
\begin{proof}
The first part is a reformulation of the last part of Proposition \ref{prop_jacquet_module}. For the second, we consider the composite
\[ \xymatrix@1{ \pi^{\mathfrak{p}} \ar[r]^-\alpha & r_Q(\pi)^{\mathfrak{p} \cap \mathfrak{m}_Q} \ar[r]^-\gamma & r_P(\pi)^{\mathfrak{m}_P}.} \]
All the maps here are isomorphisms, and we have shown the equivariance  for $\gamma$ and $\gamma \alpha$. The equivariance for $\alpha$ follows from this.
\end{proof}
\begin{proposition}\label{prop_comparison_of_Hecke}
Let $P \subset Q$ be standard parabolic subgroups of $G$, and  let $\pi $ be a  smooth $\bZ[q_v^{\pm 1/2}] [G(F_v)]$-module. 
Then for  any $x \in \pi^{\mathfrak{q}} \subset \pi^{\mathfrak{p}}$ and $s \in  (\bbZ[X_\ast(T)] \otimes_{\bbZ} \bbZ[q_v^{\pm 1/2}])^{W_{M_Q}}$, we have $\Sigma_Q(s) \cdot_{\mathfrak{q}} x = \Sigma_P(s) \cdot_{\mathfrak{p}}  x$.
\end{proposition}
\begin{proof}
We first observe that the proposition holds for the pair $P \subset Q$ of parabolic subgroups of $G$ if it holds for the pair $P \cap M_Q \subset M_Q$ of parabolic subgroups of $M_Q$. Indeed, there is a commutative diagram 
\[ \xymatrix{ \pi^{\mathfrak{q}} \ar[r] \ar[d] & r_Q(\pi)^{\mathfrak{m}_Q} \ar[d] \\
\pi^{\mathfrak{p}} \ar[r] & r_Q(\pi)^{\mathfrak{p} \cap \mathfrak{m}_Q}. } \]
The horizontal arrows are isomorphisms, by Corollary \ref{cor_compatible_with_jacquet_module}, and are equivariant with respect to the maps $\Sigma_P$ and $\Sigma_Q$. 

We next observe that the proposition holds when $P = G$ and $Q = B$. Indeed, in this case the restriction of $\Sigma_B$ to 
\[ (\bbZ[X_\ast(T)] \otimes_{\bbZ} \bbZ[q_v^{\pm 1/2}])^{W_{G}}  \subset \cH(G(F_v), \mathfrak{b}) \otimes_\bZ \bZ[q_v^{\pm 1/2}] \]
is the usual Bernstein isomorphism onto the centre of the Iwahori--Hecke algebra. The proposition in this case is the compatibility between the Bernstein isomorphism and the Satake isomorphism, cf. \cite[\S 4.6]{Hai10}.

Finally we treat the general case. By the first paragraph of the proof, we can assume $Q = G$, and allow $P$ to be an arbitrary standard parabolic subgroup. We must show that for all $x \in \pi^{\mathfrak{g}}$ and $s \in (\bbZ[X_\ast(T)] \otimes_{\bbZ} \bbZ[q_v^{\pm 1/2}])^{W_{G}}$, we have $\Sigma_G(s) \cdot_{\mathfrak{g}} x = \Sigma_P(s) \cdot_\p x$. Equivalently, we must show that $\Sigma_B(s) \cdot_{\mathfrak{b}} x = \Sigma_P(s) \cdot_\p x$. To this end we consider the commutative diagram
\[ \xymatrix{ \pi^{\mathfrak{p}} \ar[r]^-{q} \ar[d] & r_P(\pi)^{\mathfrak{m}_P} \ar[d] \\
\pi^{\mathfrak{b}} \ar[r]_-{q} & r_P(\pi)^{\mathfrak{b} \cap \mathfrak{m}_P}. } \]
where again the horizontal maps are isomorphisms. We compute
\[ q(\Sigma_B(s) \cdot_{\mathfrak{b}} x) = \Sigma_{B \cap M_P}(s) \cdot_{\mathfrak{b} \cap \mathfrak{m}_P}q(x) = \Sigma_{M_P}(s) \cdot_{\mathfrak{m}_P} q(x) = q(\Sigma_P(s) \cdot_\p x), \]
the first and third equalities by Corollary \ref{cor_compatible_with_jacquet_module} and the middle one by the current proposition for the pair $B \cap M_P \subset M_P$. Since $q$ is an isomorphism this implies that $\Sigma_B(s) \cdot_{\mathfrak{b}} x = \Sigma_P(s)\cdot_\p x$, as required. 
\end{proof}
The above proposition shows that $(\bbZ[X_\ast(T)] \otimes_{\bbZ} \bbZ[q_v^{\pm 1/2}])^{W_{G}}$ acts unambiguously on $\pi^\p$ for any standard parabolic subgroup $P \subset G$. The following corollary gives slightly more information on this action.
\begin{cor}\label{cor_centre_of_parahoric_hecke_algebra}
Let $P$ be a standard parabolic subgroup of $G$. Then $\Sigma_P((\bbZ[X_\ast(T)] \otimes_{\bbZ} \bbZ[q_v^{\pm 1/2}])^{W_{G}})$ is contained in the centre of $\cH(G(F_v), \p) \otimes_\bZ \bZ[q_v^{\pm 1/2}]$. If $\chi : T(F_v) \to \bC^\times$ is any unramified character, identified with a homomorphism $\chi : X_\ast(T) \to \bC^\times$, and $s \in (\bbZ[X_\ast(T)] \otimes_{\bbZ} \bbZ[q_v^{\pm 1/2}])^{W_{G}}$, then $s$ acts on $(i_{B(F_v)}^{G(F_v)} \chi)^{\p}$ by the scalar $\chi(s)$. 
\end{cor}
\begin{proof}
The last sentence follows from Proposition \ref{prop_comparison_of_Hecke} and the fact, already mentioned, that the restriction of $\Sigma_B$ to $(\bbZ[X_\ast(T)] \otimes_{\bbZ} \bbZ[q_v^{\pm 1/2}])^{W_{G}}$ is the usual Bernstein isomorphism to the centre of the Iwahori--Hecke algebra. The first part follows from this: if $s \in (\bbZ[X_\ast(T)] \otimes_{\bbZ} \bbZ[q_v^{\pm 1/2}])^{W_{G}}$ then $\Sigma_P(s)$ and $\Sigma_B(s) \cdot_\mathfrak{b} [\p]$ act by the same scalar on $\pi^\p$ for any irreducible admissible $\bC[G(F_v)]$-module $\pi$. This implies that they must be equal as elements of  $\cH(G(F_v), \p) \otimes_\bZ \bZ[q_v^{\pm 1/2}]$, and moreover that $\Sigma_P(s)$ lies in the centre of this Hecke algebra (as $\Sigma_B(s) \cdot_\mathfrak{b} [\p]$ does).
\end{proof}
We now specialise to our intended context. Let $G = \GL_n$ and let $P = P_{n_1, n_2}$ denote the standard parabolic associated to a partition $n = n_1 + n_2$. Let $x_1, \dots, x_n$ denote the standard basis of $X_\ast(T)$. Then we can identify $W = S_n$, $W_{M_P} = S_{n_1} \times S_{n_2}$, and 
\[ \cA = \bZ[X_\ast(T)]^W = \bZ[e_1, \dots, e_n, e_n^{-1}], \]
\[ \cB = \bZ[X_\ast(T)]^{W_{M_P}} = \bZ[a_1, \dots, a_{n_1}, b_1, \dots, b_{n_2}, e_n^{-1}], \]
where $e_1, \dots, e_n$ (resp. $a_1, \dots, a_{n_1}$, resp. $b_1, \dots,  b_{n_2}$) are the standard symmetric polynomials in $x_1, \dots, x_n$ (resp. $x_1, \dots, x_{n_1}$, resp. $x_{n_1 + 1}, \dots, x_{n_1 + n_2}$). As in  \S \ref{sec_different}, we have the resultant
\[ {\Res_{n_1, n_2}} = \prod_{i=1}^{n_1} \prod_{j=n_1 + 1}^{n_1 + n_2} (x_i - x_j) \in \bZ[X_\ast(T)]^{W_{M_P}}. \]
By Proposition \ref{prop_different}, there is a canonical lift of ${\Res_{n_1, n_2}}$ to an element $\widetilde{\Res}_{n_1, n_2} \in \cB \otimes_\cA \cB$. We record some useful properties.
\begin{lemma}\label{lem_trace_equivariance}
\begin{enumerate}
\item For any $s \in \cB$, we have $s \otimes 1 \cdot \widetilde{\Res}_{n_1, n_2} = 1 \otimes s \cdot  \widetilde{\Res}_{n_1, n_2}$.
\item Let $\pi$ be a smooth $\bZ[q_v^{\pm 1/2}][\GL_n(F_v)]$-module. Then for any $s \in \cA$, $x \in \pi^{\p}$, we have $\tr_{\p / \mathfrak{g}}(sx) = s \tr_{\p / \mathfrak{g}}(x)$.
\item Let $\pi$ be a smooth $\bZ[q_v^{\pm 1/2}][\GL_n(F_v)]$-module such that $(q_v - 1) \pi = 0$. Then for any $s \in \cB$, $x \in \pi^{\mathfrak{g}}$, we have $\tr_{\p / \mathfrak{g}}(sx) = \tr_{\cB / \cA}(s) x$.
\item Let $\pi$ be a smooth $\bZ[q_v^{\pm 1/2}][\GL_n(F_v)]$-module such that $(q_v - 1) \pi = 0$, and let $f_1, \dots, f_N$ be an $\cA$-basis for $\cB$. Write $\widetilde{\Res}_{n_1, n_2} = \sum_i f_i \otimes z_i$. Then for any $y \in \pi^{\mathfrak{p}}$, we have ${\Res_{n_1, n_2}} y = \sum_i f_i \tr_{\mathfrak{p} / \mathfrak{g}}(z_i y)$.
\end{enumerate}
\end{lemma}
\begin{proof}
The element $(s \otimes 1 - 1 \otimes s) \in \cB \otimes_\cA \cB$ lies in the kernel of the multiplication map $\cB \otimes_\cA \cB \to \cB$, so is annihilated by $\widetilde{\Res}_{n_1, n_2}$ by definition. This shows the first part. For the second, we can write $\tr_{\p / \mathfrak{g}}(sx) = [\mathfrak{g}] \cdot_\p s x$. Since $s$ lies in the centre of $\cH(G(F_v), \p) \otimes_\bZ \bZ[q_v^{\pm 1/2}]$, this equals $s [\mathfrak{g}] \cdot_\p x = s \tr_{\p / \mathfrak{g}}(x)$.

For the remaining two parts of the lemma, we fix $\pi$, a smooth $\bZ[q_v^{\pm 1/2}][\GL_n(F_v)]$-module such that $(q_v - 1) \pi = 0$. The Iwahori--Matsumoto presentation of the Iwahori--Hecke algebra descends to an isomorphism
\[ \bZ[q_v^{\pm 1/2}] / (q_v - 1)[\widetilde{W}_G] \to  \cH(G(F_v), \mathfrak{b}) \otimes_\bZ \bZ[q_v^{\pm 1/2}] / (q_v - 1). \]
Consequently, if $w \in W_G$ and $s \in \bZ[X_\ast(T)]$ then we have the identity $[\mathfrak{b} w \mathfrak{b}]\cdot_\mathfrak{b}  \Sigma_B(s) = \Sigma_B({}^w s) \cdot_\mathfrak{b} [\mathfrak{b} w \mathfrak{b}]$ in $\cH(G(F_v), \mathfrak{b}) \otimes_\bZ \bZ[q_v^{\pm 1/2} ] / (q_v - 1)$; and if $w, w' \in W_G$ then $[\mathfrak{b} w w' \mathfrak{b}] = [\mathfrak{b} w \mathfrak{b}] \cdot_\mathfrak{b} [ \mathfrak{b} w' \mathfrak{b}]$ in $\cH(G(F_v), \mathfrak{b}) \otimes_\bZ \bZ[q_v^{\pm 1/2}]  / (q_v - 1)$. If $x \in \pi^{\mathfrak{p}}$, then we have
\[ \tr_{\p / \mathfrak{g}}(x) = [\mathfrak{g}] \cdot_\mathfrak{p} x = \sum_{w \in W_G / W_{M_P}} [\mathfrak{b} w \mathfrak{b}] \cdot_{\mathfrak{b}} [\p] \cdot_{\p} x = \sum_{w \in W_G / W_{M_P}} [\mathfrak{b} w \mathfrak{b}] \cdot_{\mathfrak{b}} x. \]
This allows us to prove the third part of the lemma: if $s \in \cB$ and $x \in \pi^{\mathfrak{g}}$, we compute
\begin{multline*} \tr_{\mathfrak{p} / \mathfrak{g}}(sx) = \sum_{w \in W_G / W_{M_P}} [\mathfrak{b} w \mathfrak{b}] \cdot_{\mathfrak{b}} \Sigma_B(s) \cdot_{\mathfrak{b}} x \\ = \sum_{w \in W_G / W_{M_P}}  \Sigma_B({}^w s) \cdot_{\mathfrak{b}}  [\mathfrak{b} w \mathfrak{b}] \cdot_{\mathfrak{b}} x\\  = \Sigma_B\left( \sum_{w \in W_G / W_{M_P}} {}^w s \right) \cdot_{\mathfrak{b}} x = \Sigma_G(\tr_{\cB / \cA}(s)) \cdot_{\mathfrak{g}} x. 
\end{multline*}
For the final part, let $y \in \pi^{\mathfrak{p}}$. Then we compute
\begin{multline*} \sum_i f_i \tr_{\p / \mathfrak{g}}(z_i y) = \sum_{i} \sum_{w \in W_G / W_{M_P}} \Sigma_B(f_i) \cdot_{\mathfrak{b}} [ \mathfrak{b} w \mathfrak{b}] \cdot_{\mathfrak{b}} \Sigma_B(z_i) \cdot_{\mathfrak{b}} y 
	\\ =  \sum_{w \in W_G / W_{M_P}} \sum_{i}  \Sigma_B( f_i {}^w z_i)  \cdot_{\mathfrak{b}} [\mathfrak{b} w \mathfrak{b}] \cdot_{\mathfrak{b}}  y.
	\end{multline*}
Now Proposition \ref{prop_different} says that $ \sum_{i} f_i {}^w z_i$ equals ${\Res_{n_1, n_2}}$ if $w \in W_{M_P}$ and $0$ otherwise. We get 
\[ \Sigma_B({\Res_{n_1, n_2}}) \cdot_{\mathfrak{b}} y = {\Res_{n_1, n_2}} y. \]
This completes the proof.
\end{proof}
If $\pi$ is a smooth $\bZ[q_v^{\pm 1/2}][\GL_n(F_v)]$-module, then we define maps 
\[ f : \cB \otimes_\cA \pi^\mathfrak{g} \to \pi^\mathfrak{p}, s \otimes x \mapsto s x, \]
\[ g : \pi^\p \to \cB \otimes_\cA \pi^{\mathfrak{g}}, x \mapsto \sum_i f_i \otimes \tr_{\p / \mathfrak{g}} (z_i x), \]
where $f_1, \dots, f_N$ is an $\cA$-basis of $\cB$ and $\widetilde{\Res}_{n_1, n_2} =  \sum_i f_i \otimes z_i$. Note that $f$ is well-defined by Proposition \ref{prop_comparison_of_Hecke}. 
\begin{lemma}\label{lem_comparison_map_independent_of_choices} The map $g$  is independent of the choice of basis $f_1, \dots, f_N$. Both $f$ and $g$ are morphisms of $\cB$-modules.
\end{lemma}
\begin{proof}
Suppose that $f'_1, \dots, f'_N$ is another choice of basis. Then we can write $f'_j = \sum_i a_{ij} f_i$ for some elements $a_{ij} \in \cA$, hence $f_j = \sum_{i} b_{ij} f'_i$ for some elements $b_{ij} \in \cA$ with $\sum_k a_{ik} b_{kj} = \delta_{ij}$, hence $z'_i = \sum_j b_{ij} z_j$. We then calculate using the first part of Lemma \ref{lem_trace_equivariance}:
\begin{multline*} \sum_i f_i' \otimes \tr_{\p / \mathfrak{g}}(z_i' x) = \sum_{i,j,k} a_{ji} f_j \otimes \tr_{\p / \mathfrak{g}}(b_{ik}z_k x) \\ = \sum_{i,j,k} a_{ji} b_{ik} f_j \otimes \tr_{\p / \mathfrak{g}}(z_k x) = \sum_j f_j \otimes \tr_{\p / \mathfrak{g}}(z_j x). 
	\end{multline*}
This shows that $g$ is independent of the choice of basis. It is clear from the definition that $f$ is a morphism of $\cB$-modules. To show that $g$ is a morphism of $\cB$-modules, let $s \in \cB$, and write $s f_j = \sum_i a_{ij} f_i$ for some elements $a_{ij} \in \cB$. Then the relation given in the first part of Lemma \ref{lem_trace_equivariance} implies that $s z_i = \sum_j a_{ij} z_j$, and we compute
\begin{multline*} g(sx) = \sum_i f_i \otimes \tr_{\p / \mathfrak{g}} (z_i s x) = \sum_{i, j} f_i \otimes \tr_{\p / \mathfrak{g}} (a_{ij} z_j x) \\ = \sum_{i, j} a_{ij} f_i \otimes \tr_{\p / \mathfrak{g}} (z_j x) = \sum_j s f_j \otimes \tr_{\p / \mathfrak{g}}(z_j x) = sg(x), 	\end{multline*}
as required.
\end{proof}
\begin{proposition}\label{prop_quasi_inverse}
Suppose that $(q_v-1) \pi = 0$. Then both $fg$ and $gf$ are given by multiplication by ${\Res_{n_1, n_2}}$. Consequently, both $f$ and $g$ have the property that their kernels and cokernels are annihilated by ${\Res_{n_1, n_2}}$.
\end{proposition}
\begin{proof}
	We compute $gf$ and $fg$ in turn. First, for any element $s \otimes x \in \cB \otimes_\cA \pi^{\mathfrak{g}}$, we have
	\[ gf(x) = g(sx) = \sum_i f_i \otimes \tr_{\mathfrak{p} / \mathfrak{g}}(z_i s x). \]
	Using the third part of Lemma \ref{lem_trace_equivariance}, this becomes
	\[ \sum_i f_i \tr_{\cB / \cA}(z_i s) \otimes x. \]
	We now note the equality $\sum_i f_i \tr_{\cB / \cA}(z_i s) = \sum_i f_i z_i s = s {\Res_{n_1, n_2}}$, from which we obtain $gf(x) = {\Res_{n_1, n_2}} s \otimes x$.
	
	For the other direction, we compute
	\[ fg(y) = f \left( \sum_i f_i \otimes \tr_{\mathfrak{p} / \mathfrak{g}}(z_i y) \right) = \sum_i f_i \tr_{\mathfrak{p} / \mathfrak{g}}(z_i y). \]
	The final part of Lemma \ref{lem_trace_equivariance} is thus equivalent to the equality $fg(y) = {\Res_{n_1, n_2}} y$, as required.
\end{proof}
Now fix a prime $p$ such that $q_v \equiv 1 \text{ mod }p$. In this case we define $\Delta_v$ to be the maximal $p$-power quotient $k(v)^\times(p)$ of $k(v)^\times$. Reduction modulo $\varpi_v$, projection to second factor and determinant gives a homomorphism
\[ \p \to \GL_{n_1}(k(v)) \times \GL_{n_2}(k(v)) \to \GL_{n_2}(k(v)) \to k(v)^\times \to \Delta_v, \]
and we define $\p_1$ to be its kernel, $\ffrm_{P, 1} = \p_1 \cap M_P(F_v)$.
\begin{prop}
\begin{enumerate}
\item $\cH(M_P(F_v), \ffrm_{P, 1})$ is commutative.
\item $\cH(M_P(F_v)^+, \ffrm_{P, 1}) \subset \cH(M_P(F_v), \ffrm_{P, 1})$ is a subalgebra.
\item The homomorphism of $\bZ[q_v^{\pm 1/2}]$-modules $T_{P, 1}^+ : \cH(M_P(F_v)^+, \ffrm_{P, 1}) \to \cH(G(F_v), \p_1)$ defined on basis elements by 
\[ T_{P, 1}^+([\ffrm_{P, 1} m \ffrm_{P, 1}]) = \delta_P(m)^{1/2} [\p_1 m \p_1] \]
is an injective algebra homomorphism, with image equal to the set of functions with support in $\p_1 M_P(F_v)^+ \p_1$.
\item $T_{P, 1}^+$ extends uniquely to an algebra homomorphism $T_{P, 1} : \cH(M_P(F_v), \ffrm_{P, 1}) \otimes_\bZ \bZ[q_v^{\pm 1/2}] \to \cH(G(F_v), \p_1) \otimes_\bZ \bZ[q_v^{\pm 1/2}]$.
\item Let $\pi$ be a smooth $\bZ[q_v^{\pm 1/2}] [G(F_v)]$-module, and let $q : \pi^{\mathfrak{p}_1} \to r_P(\pi)^{\mathfrak{m}_{P, 1}}$ denote the canonical projection. Then $q$ is an isomorphism and for any $x \in \pi^{\mathfrak{p}_1}$, $t \in \cH(M_P(F_v), \mathfrak{m}_{P, 1}) \otimes_\bZ \bZ[q_v^{\pm 1/2}]$, we have $q(T_{P, 1}(t) x) = t q(x)$.
\end{enumerate}
\end{prop}
\begin{proof}
$\cH(M_P(F_v), \ffrm_{P, 1})$ is commutative by Gelfand's trick: there is a set of double coset representatives for $\ffrm_{P, 1} \backslash M_P(F_v) / \ffrm_{P, 1}$ which is invariant under $g \mapsto {}^t g$ (we can take the matrices of the form 
\[ \diag( \varpi_v^{k_1}, \varpi_v^{k_2}, \dots, \varpi_v^{k_{n_1}}, \alpha \varpi_v^{k_{n_1+1}}, \varpi_v^{k_{n_1+2}}, \dots, \varpi_v^{k_n}), \]
 where $k_1 \geq k_2 \geq \dots \geq k_{n_1}$, $k_{n_1+1} \geq \dots \geq k_{n}$, and $\alpha$ ranges over a set of representatives for the quotient $\Delta_v$ of $\cO_{F_v}^\times$). The proof of the remainder of the proposition is basically the same as the proof of the corresponding parts of Proposition \ref{prop_jacquet_module}, provided we can exhibit a strongly $(P, \p_1)$-positive element $z \in Z(M_P)(F_v)$ such that $[\p_1 z \p_1]$ is invertible. In fact, the result of Lemma \ref{lem_inverse_to_Hecke} holds with $\p$ replaced by $\p_1$, with essentially the same proof, using \cite[Corollary 1]{Vig05}.
\end{proof}
We have defined a map $\Sigma_P : \cB \otimes_\bZ \bZ[q_v^{\pm 1/2}] \to \cH(G(F_v), \p) \otimes_\bZ \bZ[q_v^{\pm 1/2}]$ using $T_P$ and the Satake isomorphism. We define a map $\Sigma_{P, 1} : \cB[q_v^{\pm 1/2}, \Delta_v] \to \cH(G(F_v), \p_1) \otimes_\bZ \bZ[q_v^{\pm 1/2}]$ as follows: it is the composite with $T_{P, 1}$ of the tensor product of the homomorphisms $\cB\otimes_\bZ \bZ[q_v^{\pm 1/2}] \to \cH(M_P(F_v), \ffrm_{P, 1})\otimes_\bZ \bZ[q_v^{\pm 1/2}]$, $\bZ[\Delta_v] \to \cH(M_P(F_v), \ffrm_{P, 1})$ given by the formulae
\[ a_i \mapsto q_v^{i(i-n_1)/2} [\ffrm_{P, 1} \diag(\underbrace{\varpi_v, \dots, \varpi_v}_i, 1, \dots, 1) \ffrm_{P, 1}] \]
\[ b_i \mapsto q_v^{i(i-n_2)/2} [\ffrm_{P, 1} \diag(\underbrace{1, \dots, 1}_{n_1}, \underbrace{\varpi_v, \dots, \varpi_v}_i, 1, \dots, 1) \ffrm_{P, 1}] \]
and
\[ \alpha \in \Delta_v \mapsto \langle \alpha \rangle =  [\ffrm_{P, 1} \diag(\underbrace{1, \dots, 1}_{n_1}, \alpha, 1, \dots, 1) \ffrm_{P, 1}]. \]
If $\pi$ is a smooth $\bZ[q_v^{\pm 1/2}][\GL_n(F_v)]$-module then we use $\Sigma_{P, 1}$ to view $\pi^{\p_1}$ as an $\cB[q_v^{\pm 1/2}, \Delta_v]$-module.
\begin{lemma}\label{lem_restriction_to_P_invariants}
Let $\pi$ be a smooth $\bZ[q_v^{\pm 1/2}][\GL_n(F_v)]$-module. Then $\pi^{\mathfrak{p}} \subset \pi^{\mathfrak{p}_1}$ is an $\cB[\Delta_v]$-submodule on which $\Delta_v$ acts trivially, and the induced structure of $\cB$-module agrees with the one induced by $\Sigma_P$.
\end{lemma}
\begin{proof}
It is clear from the definitions that $\Delta_v$ acts trivially on $\pi^{\mathfrak{p}}$. What needs to be checked is that e.g.\ the two operators
\[ [\p_1 \diag(\underbrace{\varpi_v, \dots, \varpi_v}_i, 1, \dots, 1) \p_1],\, [\p \diag(\underbrace{\varpi_v, \dots, \varpi_v}_i, 1, \dots, 1) \p] \]
defining the action of $a_i$ act in the same way on $\pi^{\mathfrak{p}}$ (and similarly for the operators defining the action of $b_i$). This is true because, writing 
\[ \eta_i = \diag(\underbrace{\varpi_v, \dots, \varpi_v}_i, 1, \dots, 1), \]
the maps $\p_1 \eta_i \p_1 / \p_1 \to \p \eta_i \p / \p$ are bijections.
\end{proof}
Let $\pi$ be an irreducible admissible $\bC[\GL_n(F_v)]$-module. Suppose that $\rec_{F_v}(\pi) = \oplus_{i=1}^r \operatorname{Sp}_{m_i}(\chi_i \circ \Art_{F_v}^{-1})$, where $\chi_i : F_v^\times \to \bC^\times$ are smooth characters and $\operatorname{Sp}_m = (r_m, N_m)$ is the Weil--Deligne representation given by $r_m = \oplus_{i=1}^m | \cdot|^{(m+1-2i)/2} \circ \Art_{F_v}^{-1}$, $N_m e_i = e_{i-1}$ if $i > 0$, $e_1, \dots, e_m$ the standard basis of $\bC^m$. Then $\pi$ is isomorphic to a subquotient of the induced representation
\[ \Pi = i_{P_{m_1, \dots, m_r}(F_v)}^{\GL_n(F_v)} \otimes_{i=1}^r \St_{m_i}(\chi_i), \]
where $P_{m_1, \dots, m_r}$ is the standard parabolic subgroup of $\GL_n$ corresponding to the partition $n = m_1 + m_2 + \dots + m_r$.
\begin{proposition}\label{prop_trace_relation_0}
Let $\pi$ be an irreducible admissible $\bC[\GL_n(F_v)]$-module. If $\Res_{q_v, n_1, n_2}^{n!} \pi^{\mathfrak{p}} \neq 0$, then $\pi$ is unramified and $\pi^{\mathfrak{p}} = \pi^{\mathfrak{p}_1}$.
\end{proposition}
\begin{proof}
Since $\pi^{\mathfrak{p}} \neq 0$, we have in particular $\pi^{\mathfrak{b}} \neq 0$, so there is an isomorphism $\rec_{F_v}(\pi)  = \oplus_{i=1}^r \operatorname{Sp}_{m_i}(\chi_i \circ \Art_{F_v}^{-1})$, where the characters $\chi_i$ are unramified, and $\pi$ is isomorphic to a subquotient of the representation $\Pi$ as above.	We compute the Jacquet module $r_P(\Pi)$. According to the `geometrical lemma' \cite[Lemma 2.12]{Bel77} and \cite[Lemma 5.1]{jack}, $r_P(\Pi)$ admits a filtration whose graded pieces $\sigma_\lambda$ are indexed by decompositions $m_j = \lambda_{1j} + \lambda_{2j}$, $j = 1, \dots, r$, where $\lambda_{ij}$ are non-negative integers such that $\sum_j \lambda_{1j} = n_1$ and $\sum_j \lambda_{2j} = n_2$. The representation $\sigma_\lambda$ can be described as follows: let $P_{\lambda, i}$ denote the standard parabolic subgroup of $\GL_{n_i}$ associated to the decomposition $n_i = \lambda_{i1} + \dots + \lambda_{ir}$. Then we have
\begin{multline}\label{eqn_graded_piece_of_Jacquet_module}  \sigma_\lambda = \left( i_{P_{\lambda, 1}(F_v)}^{\GL_{n_1}(F_v)} \otimes_{j=1}^r \St_{\lambda_{1j}}(|\cdot|^{(m_j - \lambda_{1j})/2} \psi_j) \right) \\ \otimes \left( i_{P_{\lambda, 2}(F_v)}^{\GL_{n_2}(F_v)} \otimes_{j=1}^r \St_{\lambda_{2j}}(|\cdot|^{(\lambda_{2j} - m_j)/2} \psi_j) \right). \end{multline}
Since passage to invariants under an open compact subgroup is exact, Proposition \ref{prop_jacquet_module} implies that $\Pi^{\p} \neq 0$ if and only if $\sigma_\lambda^{\mathfrak{m}_P} \neq 0$ for some $\lambda$, or in other words if there exists a decomposition $m_j = \lambda_{1j} + \lambda_{2j}$ ($j = 1, \dots, r$) such that $\lambda_{ij} = 0$ or $1$ for all $i, j$. This implies that $m_j \leq 2$ for all $j$. Suppose that $m_j = 2$ for some $j$ and that $\lambda_{1j} = \lambda_{2j} = 1$. Then $\Res_{q_v, n_1, n_2}$ acts on $\sigma_\lambda^{\mathfrak{m}_P}$ by a scalar which is divisible by $(q_v \psi_j(\varpi_v)|\varpi_v|^{1/2} - \psi_j(\varpi_v) |\varpi_v|^{-1/2} ) = 0$. Since the dimension of $\Pi^{\p}$ as $\bC$-vector space is bounded above by $n!$, we conclude that if $m_j = 2$ for some $j$ then $\Res_{q_v, n_1, n_2}^{n!} \Pi^{\mathfrak{p}}  = 0$, contradicting our hypothesis.

We conclude that $m_j = 1$ for all $j$ and therefore that $\pi$ is unramified (since $\rec_{F_v}(\pi)$ is). It remains to explain why $\pi^\p = \pi^{\p_1}$. Since passage to invariants under an open compact subgroup is exact, it's enough to show that $\Pi^{\p} = \Pi^{\p_1}$ or even that $\sigma_\lambda^{\mathfrak{m}_P} = \sigma_{\lambda}^{\mathfrak{m}_{P, 1}}$ for each $\lambda$. This is true.
\end{proof}
\begin{cor}\label{cor_trace_relation_0}
Let $p$ be a prime such that $(p, q_v) = 1$ and fix an isomorphism $\iota : \overline{\bQ}_p \to \bC$ (and hence a $\bZ[q_v^{\pm 1/2}]$-algebra structure on $\overline{\bQ}_p$). Let $\pi$ be an irreducible admissible $\overline{\bQ}_p[\GL_n(F_v)]$-module such that $\pi^{\p} \neq 0$. Suppose given a continuous homomorphism $\rho : G_{F_v} \to \GL_n(\overline{\bQ}_p)$ such that $\mathrm{WD}(\rho)^{F-ss} \cong \rec_{F_v}^T(\pi_v)$. Then either $\Res_{q_v, n_1, n_2}^{n!} \pi^{\p} = 0$ or $\rho$ is unramified.
\end{cor}
\begin{prop}\label{prop_trace_relation_1}
Let $\pi$ be an irreducible admissible $\bC[\GL_n(F_v)]$-module. Suppose that $\pi^{\mathfrak{p}} = 0$ but $\pi^{\mathfrak{p}_1} \neq 0$. Then $\pi^{\mathfrak{p}_1}$ has dimension 1. If $\rec_{F_v}(\pi) = (r, N)$, then $N = 0$ and there is an isomorphism $r = \oplus_{i=1}^n \chi_i \circ \Art_{F_v}^{-1}$, where the characters $\chi_1, \dots, \chi_{n_1}$ are unramified and the characters $\chi_{n_1+1}, \dots, \chi_n$ are ramified with equal restriction to $\cO_{F_v}^\times$. The algebra $\cB$ acts on $\pi^{\p_1}$ according to the factorisation $\det(X - r(\phi_v)) = F_1(X) F_2(X)$, where $F_1(X) = \prod_{i=1}^{n_1}(X - \chi_i(\varpi_v))$ and $F_2(X) = \prod_{j=1}^{n_2} (X - \chi_{n_1 + j}(\varpi_v))$, and the group $\Delta_v$ acts on $\pi^{\p_1}$ according to the character $\langle \alpha \rangle \mapsto \chi_{n}(\alpha)$.

Finally, let $f_\pi : \cB[\Delta_v] \to \bC$ be the character giving the action of $\cB[\Delta_v]$ on $\pi^{\p_1}$. Then for every pair $\tau \in I_{F_v}$, $\alpha \in \cO_{F_v}^\times$ such that $\alpha = \Art_{F_v}^{-1}(\tau)$,  we have
\begin{equation}\label{eqn_trace_relation_1}
f_\pi(\Res_{n_1, n_2}^2) r(\tau) = f_\pi(\Res_{n_1, n_2})(e_1(r(\phi_v)) + f_\pi(\langle \alpha \rangle) e_2(r(\phi_v))).
\end{equation}
\end{prop}
The polynomials $e_1(X), e_2(X) \in \cB[X]$ appearing in (\ref{eqn_trace_relation_1}) are the ones defined in (\ref{eqn_projectors}).
\begin{proof} The argument is similar to the proof of Proposition \ref{prop_trace_relation_0}. By \cite[Lemma 3.1.6]{cht}, there exist characters $\chi_1, \dots, \chi_n : F_v^\times \to \bC^\times$ such that $\chi_1, \dots, \chi_{n_1}$ are unramified, $\chi_{n_1 +1 }, \dots, \chi_n$ are tamely ramified with equal restriction to inertia, and such that $r = \oplus_{i=1}^n \chi_i \circ \Art_{F_v}^{-1}$. We can also write $(r, N) = \oplus_{j=1}^t \operatorname{Sp}_{m_j}(\psi_j \circ \Art_{F_v}^{-1})$ for tamely ramified characters $\psi_j : F_v^\times \to \bC^\times$, so that $\pi$ is a subquotient of the induced representation
\[ \Pi = i_{P_{m_1, \dots, m_t}}^{\GL_n(F_v)} \otimes_{j=1}^t \St_{m_j}(\psi_j). \]
As in the proof of Proposition \ref{prop_trace_relation_0} we see that $r_P(\Pi)$ admits a filtration with graded pieces $\sigma_\lambda$ indexed by decompositions $m_j = \lambda_{1j} + \lambda_{2j}$ with $\lambda_{ij}$ non-negative integers such that $\sum_j \lambda_{ij} = n_i$, and $\sigma_\lambda$ given by the equation (\ref{eqn_graded_piece_of_Jacquet_module}). Since $\mathfrak{m}_{P, 1}$ contains $\GL_{n_1}(\cO_{F_v})$, we see that $\sigma_\lambda^{\mathfrak{m}_{P, 1}}$ can be non-zero only if $\lambda_{ij} \leq 1$ for each $i, j$ and moreover that if $\lambda_{1j} = 1$ then $\psi_j$ is unramified.

Fix $\lambda$ such that $\sigma_\lambda^{\mathfrak{m}_{P, 1}} \neq 0$. We see that if $m_j = 2$ (hence $\lambda_{1j} = \lambda_{2j} = 1$) then $\psi_j$ is unramified, hence all characters $\psi_k$ must be unramified, hence $\sigma_\lambda^{\mathfrak{m}_P} \neq 0$. It follows that $\Pi^{\p} = \Pi^{\p_1}$, hence $\pi^{\p} = \pi^{\p_1} \neq 0$, contradicting our hypothesis.

We conclude that $m_j = 1$ for all $j$, or in other words that $N = 0$. Thus $t = n$, and we can assume that $\psi_j = \chi_j$. Then there is a unique choice of $\lambda$ for which $\sigma_\lambda^{\mathfrak{m}_{P, 1}} \neq 0$, namely $(\lambda_{1j}, \lambda_{2j}) = (1, 0)$ if $j = 1, \dots, n_1$ and $(0, 1)$ if $j = n_1 + 1, \dots, n$. This shows that $\Pi^{\p_1}$ is 1-dimensional, hence that $\pi^{\p_1}$ is 1-dimensional (since it is assumed non-zero).

It remains to establish the formula (\ref{eqn_trace_relation_1}). We split into cases. If $f_\pi(\Res_{n_1, n_2}) = 0$ then both sides are zero. If $f_\pi(\Res_{n_1, n_2}) \neq 0$ then $e_1(r(\phi_v))$ is $f_\pi(\Res_{n_1, n_2})$ times the idempotent which projects to the subrepresentation $\oplus_{i=1}^{n_1} \chi_i$ of $r$, and similarly for $e_2(r(\phi_v))$, in which case the formula follows from the fact that the characters $\chi_1, \dots, \chi_{n_1}$ are unramified and the characters $\chi_{n_1 +1}, \dots, \chi_n$ have the same restriction to $\cO_{F_v}^\times$.
\end{proof}
\begin{cor}\label{cor_trace_relation_1}
Let $p$ be a prime such that $(p, q_v) = 1$ and fix an isomorphism $\iota : \overline{\bQ}_p \to \bC$ (and hence a $\bZ[q_v^{\pm 1/2}]$-algebra structure on $\overline{\bQ}_p$). Let $\pi$ be an irreducible admissible $\overline{\bQ}_p[\GL_n(F_v)]$-module. Suppose that $\pi^{\mathfrak{p}} = 0$ but $\pi^{\mathfrak{p}_1} \neq 0$, and suppose given a continuous homomorphism $\rho : G_{F_v} \to \GL_n(\overline{\bQ}_p)$ such that $\mathrm{WD}(\rho)^{F-ss} \cong \rec_{F_v}^T(\pi_v)$. Then $\pi^{\p_1}$ is 1-dimensional; let $f_\pi : \cB[\Delta_v] \to \overline{\bQ}_p$ be the character by which this algebra acts on $(\pi | \cdot |^{(1-n)/2})^{\p_1}$. Then for every pair $\tau \in I_{F_v}$, $\alpha \in \cO_{F_v}^\times$ such that $\alpha = \Art_{F_v}^{-1}(\tau)$, we have
\[ f_\pi(\Res_{n_1, n_2}^2) \rho(\tau) = f_\pi(\Res_{n_1, n_2}) (e_1(\rho(\phi_v)) + f_\pi(\langle \alpha \rangle) e_2(\rho(\phi_v)). \]
\end{cor}
\begin{proof}
 If $f_\pi(\Res_{n_1, n_2}) = 0$ then both sides of the proposed equality are zero, so we can assume that $f_\pi(\Res_{n_1, n_2}) \neq 0$. In this case we write $\mathrm{WD}(\rho)^{F-ss} = (r, N)$, where $N = 0$ and, if $\rho(\phi_v) = su$ is the multiplicative Jordan decomposition, then $r(\phi_v) = s$. The result will follow from Proposition \ref{prop_trace_relation_1} if we can show that $e_i(\rho(\phi_v)) = e_i(s)$. This is true, since $s$ and $su$ have the same generalised eigenspaces.
\end{proof}

\section{Weak adequacy in characteristic 0}\label{sec_adequacy}

In this section, let $k$ be a field of characteristic 0.
\begin{lemma}
Let $G$ be a linear algebraic group over $k$ such that $G^0$  is reductive. Then we can find a dense open subset $U \subset G$ consisting entirely of semisimple elements.
\end{lemma}	
\begin{proof}
We are free to replace $k$ by a finite extension, and can assume that each connected component of $G$ has a rational point. Then it suffices to construct for each $h \in G(k)$ a dense open subset $U_h \subset G^0 h$ consisting entirely of semisimple elements. The unipotent part of $h$ is in $G^0$, so we can assume that $h$ is semisimple. Then $\Ad(h)$ is a semisimple automorphism of $G^0$, so \cite[Theorem 7.5]{Ste68} implies that, after possibly further enlarging $k$, we can find a split maximal torus and Borel subgroup $T \subset B \subset G^0$ which are invariant under $\Ad(h)$.

Let $S = Z_T(h)^\circ$. We define a map $\mu : G^0 \times S \to G^0 h$, $(g, s) \mapsto g s h g^{-1} = g s \Ad(h)(g^{-1}) h$. We claim that the image of $\mu$ is dense in $G^0 h$. This will imply the lemma: the image of $\mu$ is constructible, so contains a dense open subset of $G^0 h$. The image of $\mu$ consists of semisimple elements, since $S h$ consists of semisimple elements.

To prove the claim, it is enough to exhibit $s \in S(k)$ such that the centralizer in $\Lie G^0$ of $\Ad(sh)$ is $\Lie S$. Indeed, then computing the differential shows that $\mu$ is smooth in a neighbourhood of $(1, s)$. The existence of an $s$ with this property can be read off  from \cite[Proposition 3.8]{Ree10}.
\end{proof}
In the statement of the next result, we write $h = h^{ss} h^u$ for the multiplicative Jordan decomposition of an element $h \in \GL_n(k)$.
\begin{lemma}\label{lem_weakly_adequate}
Let $H \subset \GL_n(k)$ be a subgroup, and suppose that for each $h \in H$, the characteristic polynomial of $h$ splits into linear factors over $k$. Then the following are equivalent:
\begin{enumerate}
\item The span of the set $\{ h^{ss} \mid  h \in  H \}$ equals $M_n(k)$.
\item For every non-zero $H$-invariant subspace $W \subset M_n(k)$, there exists $h \in H$ and an eigenvalue $\alpha \in k$ of $h$ such that $\tr e_{h, \alpha} W \neq 0$ (where $e_{h, \alpha}$ projects to the generalised $\alpha$-eigenspace of $h$).
\end{enumerate}
\end{lemma}
\begin{proof}
For a given subspace $W \subset M_n(k)$, the existence of $h, \alpha$ such that $\tr e_{h, \alpha} W \neq 0$ is equivalent to  the existence of  an element $h \in H$ such that $\tr h^{ss} W \neq 0$ (as $e_{h, \alpha}$ is a polynomial  in  $h^{ss}$).
\end{proof}
When $k$ has characteristic $p$,  Guralnick \cite{Gur12} calls subgroups satisfying the analogue of the equivalent conditions of Lemma  \ref{lem_weakly_adequate} ``weakly adequate''. The following (easy) proposition shows that when $k$ has characteristic 0, this condition is equivalent to absolute irreducibility.
\begin{proposition}\label{prop_weak_adequacy}
Let $H \subset \GL_n(k)$ be a subgroup which is absolutely irreducible. Then the span of the set $\{ h \in H \mid h = h^{ss} \}$ equals $M_n(k)$.
\end{proposition}
\begin{proof}
 Let $G$ be the Zariski closure of $H$ in $\GL_n$, and let $U \subset G$ be a dense open subset consisting of semisimple elements. Then $U \cap H \subset \{ h \in H \mid h = h^{ss} \}$ and $U \cap H$ is Zariski dense in $G$. If the span  of  $U \cap H$ does not equal $M_n(k)$, then $G$ is contained  in a proper linear subspace of $M_n(k)$, hence so is $H$. This contradicts Burnside's lemma.
\end{proof}
We conclude this section by giving some examples of subgroups of $\GL_n(k)$ which are irreducible but not enormous, in the sense of \cite[Definition 2.27]{New19a}, along similar lines to the examples of non-big subgroups given by Barnet-Lamb \cite[\S 5.2]{barnetlamb2010nonbig}. This shows that the results of this paper really are stronger than those of \cite{New19a}. 

It is easy to give examples of finite irreducible subgroups of $\GL_n(k)$ containing no regular semisimple element (for example, the image of the 10-dimensional irreducible representation of $A_6$). The definition implies that such subgroups can not be enormous. Such examples are less relevant to the context considered here, since we are interested in the images of the Galois representations attached to regular algebraic automorphic representations; Sen theory implies that images of such representations should always contain regular semisimple elements, so we need to consider the interaction with the decomposition of the adjoint representation.

To this end, let $H' \subset \GL_2(k)$ denote the normalizer of the group of diagonal matrices, and let $H$ denote the image of $H' \times H'$ under the tensor product representation $\GL_2 \times \GL_2 \to \GL_4$. One can check that $H$ is absolutely irreducible but not enormous, because the span of the regular semisimple elements of $H$ in $M_4(k)$ is contained in the subspace of matrices with 0's on the anti-diagonal.

\section{Galois pseudodeformation theory}\label{sec_deformation}

Let us suppose given the following data:
\begin{itemize}
\item A prime $p$, a finite extension $E / \bQ_p$ inside the fixed algebraic closure $\overline{\bQ}_p$, and an isomorphism $\iota : \overline{\bQ}_p \to \bC$. We assume that $E$ contains all quadratic extensions of $\bQ_p$, so that using $\iota$, $\cO$ has a canonical structure of $\bZ[q^{\pm 1/2}]$-algebra for any prime number $q \neq p$.
\item A CM field $F$ with maximal totally real subfield $F^+$.
\item A finite set $S$ of finite places of $F^+$, including the set $S_p$ of $p$-adic places, which all split in $F$.
\item For each $v \in S$, a factorisation $v = \wv \wv^c$ in $F$. We write $\widetilde{S}$ for the set  of places $\wv$.
\item  A continuous representation $r  : G_{F^+, S} \to \cG_n(\cO)$ such that $\rho = r|_{G_{F, S}} \otimes_\cO  E$ is absolutely irreducible and $\nu \circ r = \delta_{F / F^+}^n \epsilon^{1-n}$.
\item Integers $a \leq b$ such that all of the Hodge--Tate weights of $\rho$ lie in the interval $[a, b]$ and $a + b = n-1$.
\end{itemize}
In this section, we will write $\mathcal{DET}(\sigma)$ for the group determinant (in the sense of \cite{chenevier_det}) associated to a representation $\sigma$. Let $\overline{D} = \mathcal{DET}(\overline{\rho})$ denote the group determinant of $G_{F, S}$ associated to $\overline{\rho}$, and let $R_S \in \cC_\cO$ denote the object representing the functor of conjugate self-dual deformations of $\overline{D}$ that are unramified outside $S$ and semistable with Hodge--Tate weights in $[a, b]$, as defined in \cite[\S 2.19]{New19a}.

We define $W_\cO = \ad r$, $W_m = W_\cO / (\varpi^m)$, $W_E = \ad r \otimes_\cO E$, and $W_{E / \cO} = W_\cO \otimes_\cO E / \cO$; these are $\cO[G_{F^+, S}]$-modules. We write $\cL_S = \{ \cL_{v, m} \}$ for the Selmer conditions for $W_m$ defined as in \cite[\S 2.19]{New19a} (semistable with Hodge--Tate weights in $[a, b]$ at places above $p$, unramified outside $S$, no restriction at places of $S - S_p$). We write $D_S$ for the universal group determinant over $R_S$ and  $\Lambda_i : G_{F, S} \to R_S$ for the coefficients of the universal characteristic polynomial $D_S(X - \sigma) = \sum_{i=0}^n (-1)^i \Lambda_i(\sigma) X^{n-i}$.

We define a Taylor--Wiles datum $\scrQ = (Q, \widetilde{Q}, (f_{v, 1}(X) )_{v \in Q}, (f_{v, 2}(X) )_{v \in Q})$ of level $N \geq 1$ to be a tuple consisting of the following data:
\begin{itemize}
\item A tuple $Q = (v_1, \dots, v_q)$ of distinct finite places of $F^+$ such that for each $i = 1, \dots, q$, $v_i\not\in S$, $v_i$ splits in $F$, and $q_{v_i} \equiv 1 \text{ mod }p^N$.
\item A tuple $(\widetilde{v}_1, \dots, \widetilde{v}_q)$ of finite places of $F$ such that $\widetilde{v}_i$ lies above $v_i$.
\item For each $i = 1, \dots, q$, a factorisation $f_{v_i}(X) := \det(X - \rho(\Frob_{\wv_i})) = f_{v_i, 1}(X) f_{{v_i}, 2}(X)$ in $\cO[X]$, where $f_{v, 1}(X), f_{v, 2}(X)$ are monic polynomials with no common roots in $\overline{\bQ}_p$. 
\end{itemize}
If $\scrQ$ is a Taylor--Wiles datum and $v \in Q$, then we define $\Delta_v$ to be the maximal $p$-power quotient of $k(\wv)^\times$ and $\Delta_Q = \prod_{v \in Q} \Delta_v$. If $\tau \in I_{F_\wv}$, we write $\langle \tau \rangle \in \Delta_v$ for the image of $\Art_{F_\wv}^{-1}(\tau)$ in $\Delta_v$. We write $t_v : I_{F_\wv} \to \bZ_p$ for any choice of surjective homomorphism. We define $\cA(\scrQ) = \otimes_{i=1}^q \cA $ and $\cB(\scrQ) = \otimes_{i=1}^q \cB_{\deg f_{v, 1}, \deg f_{v, 2}}$, where $\cA, \cB$ are as considered in \S \ref{sec_different}.

 We define an enhancement  $R(\scrQ)$ of the universal deformation ring $R_{S \cup Q}$ as follows. It will be a complete Noetherian semi-local $\cO$-algebra. If $v \in Q$, let $F_v(X) = D_{S \cup Q}(X - \phi_\wv) \in R_{S \cup Q}[X]$ be the characteristic polynomial of a fixed Frobenius lift $\phi_\wv$ in the universal deformation. The polynomials $F_v(X)$ $(v \in Q)$ give $R_{S \cup Q}$ the structure of $\cA(\scrQ)$-algebra. Over the ring $R_{S \cup Q} \otimes_{\cA(\scrQ)} \cB(\scrQ)$, we have universal factorisations $F_v(X) = F_{v, 1}(X) F_{v, 2}(X)$, where $F_{v, 1}(X)$, $F_{v, 2}(X)$ are monic polynomials of degrees $\deg f_{v, 1}$, $\deg f_{v, 2}$, respectively, and (after Lemma \ref{lem_universal_idempotents}) polynomials $e_{v, 1}(X)$, $e_{v, 2}(X)$ such that $e_{v, 1}(X) + e_{v, 2}(X) = \operatorname{Res}_v$, where we write $\Res_v$ for the image of $R_{\deg f_{v, 1}, \deg f_{v, 2}}$ in $\cB(\scrQ)$. We also write $\Res_{v, q}$ for the image of $R_{q_v, \deg f_{v, 1}, \deg f_{v ,2}}$ in $\cB(\scrQ)$. We define $R(\scrQ)$ be the quotient of 
\[R_{S \cup Q} \otimes_{\cA(\scrQ)} \cB(\scrQ) \otimes_\cO \cO[\Delta_Q] \]
defined by the relation 
\begin{equation}\label{eqn_trace_relations} \Res_{v, q}^{n!} \Lambda_1( \sigma( \Res_v^2 \tau - \Res_v e_{v, 1}(\phi_\wv) - \langle \tau \rangle \Res_v e_{v, 2}(\phi_\wv) ) ) = 0
\end{equation}
for all $v \in Q$, $\tau \in I_{F_\wv}$, $\sigma \in G_{F, S \cup Q}$. We write $P(\mathscr{Q}) \subset R(\scrQ)$ for the kernel of the homomorphism 
\[ f_\scrQ :   R(\scrQ) \to R_{S}  \otimes_{\cA(\scrQ)} \cB(\scrQ) \to \cO \]
associated to $\mathcal{DET}(\rho)$ and the fixed factorisations $f_v(X) = f_{v, 1}(X) f_{v, 2}(X)$ ($v \in Q$).

If $\scrQ$ is a Taylor--Wiles datum of level $N \geq 1$ and $m$ is an integer such that $1 \leq m \leq N$, then we define modified local conditions $\cL(\scrQ) = \{ \cL(\scrQ)_{v, m} \}$ for the $\cO[G_{F^+, S \cup Q}]$-module $W_m$ as follows: if  $v \not\in Q$, then $\cL(\scrQ)_{v, m} = \cL_{v, m}$. If $v \in Q$, then we define $\cL_{v, m}$ to be the pre-image (under restriction) in $H^1(F^+_v, W_m)$ of the $\cO$-submodule of
\[  H^1(I_{F_\wv}, W_m)^{G_{F_\wv}} \cong \Hom_{cts}(I_{F_\wv}, W_m^{G_{F_\wv}}), \]
generated by the homomorphism
\[ \tau  \mapsto  t_v(\tau) e_{v, 2}(\rho(\Frob_\wv)) \text{ mod }\varpi^m. \]
(In interpreting this, we point out that $e_{v, 2}(\rho(\Frob_\wv)) \in M_n(\cO)$ is $f_\scrQ(\Res_v)$ times the idempotent in $M_n(E)$ which projects to the sum of the $\alpha$-generalised eigenspaces of $\rho(\Frob_\wv)$ for those $\alpha$ with $f_{v, 2}(\alpha) = 0$; moreover, the definition of $\cL_{v, m}$ is independent of the choice of homomorphism $t_v$.) We write $l(\scrQ)_{v, m}$ for the length of $\cL(\scrQ)_{v, m}$ as $\cO$-module. We write $\cL(\scrQ)^\perp = \{ \cL(\scrQ)_{v, m}^\perp \}$ for the dual local conditions for the $\cO[G_{F^+, S \cup Q}]$-module $W_m(1)$.
\begin{lemma}\label{lem_comparing_Selmer_and_tangent_space}
 There exists a constant $d \geq 0$ with the following property: for any $m \geq 1$ and for any Taylor--Wiles datum $\scrQ$ of level $N \geq m$, there exists a homomorphism of $\cO$-modules
\[  H^1_{\cL(\scrQ)}(F^+, W_m) \to \Hom_\cO(P(\scrQ) / P(\scrQ)^2, \cO  / \varpi^m  \cO) \]
with kernel and cokernel annihilated  by $\varpi^d \Res(\scrQ)^{3+n!}$, where we define
\[ \Res(\scrQ) =  \operatorname{lcm}( \{ f_\scrQ(\Res_v)  \}_{v \in Q}) \in \cO. \]
\end{lemma}
\begin{proof}
We can identify $\Hom_\cO(P(\scrQ) / P(\scrQ)^2, \cO  / \varpi^m  \cO)$ with the set of $\cO$-algebra morphisms $R(\scrQ) \to \cO \oplus \epsilon \cO  / \varpi^m  \cO$ which recover $f_\scrQ$ after reduction modulo $\epsilon$. Let $[\phi] \in H^1_{\cL(\scrQ)}(F^+, W_m)$. We associate to $\phi$ a homomorphism  $\rho_\phi : G_{F, S \cup Q} \to \GL_n(A_m)$ by the formula $\rho_\phi(\sigma) = \rho(\sigma)(1 + \epsilon \phi(\sigma))$. If $v \in Q$, let $f_{\phi, v}(X) = \det(X - \rho_\phi(\phi_\wv)) \in A_m[X]$. Using Lemma \ref{lem_approximate_hensel}, we are given a factorisation $f_{\Res(\scrQ) \phi, v}(X) =  \alpha_{\Res(\scrQ)}(f_{\phi, v}(X)) = f_{\Res(\scrQ) \phi, v, 1}(X) f_{\Res(\scrQ) \phi, v, 2}(X)$ in $A_m[X]$ lifting the factorisation $f_{v, X}  = f_{v, 1}(X) f_{v, 2}(X)$ in $\cO[X]$. There exists a constant $\lambda_v \in \cO$ such that $\phi(\tau) = \lambda_v t_v(\tau) e_{v, 2}(\rho(\Frob_\wv))$ for all $\tau \in I_{F_\wv}$, and we define a homomorphism $\Delta_v \to 1 + \epsilon \cO / \varpi^m \cO$ by $\tau \mapsto 1 + \epsilon \Res(\scrQ) \Res_v \lambda_v t_v(\tau)$ (this depends only on $\phi$ and not on the choice of $\lambda_v$). With the group determinant $\mathcal{DET}(\rho_{\Res(\scrQ) \phi})$, these data  define a homomorphism $R_{S \cup Q} \otimes_{\cA(\scrQ)} \cB(\scrQ) \otimes_\cO \cO[\Delta_Q] \to A_m$. We claim that it factors through the quotient $R(\scrQ)$. It is enough to show the equality
\begin{multline*} \Res_v^2 ( 1 + \epsilon \Res(\scrQ) \phi(\tau) ) \\ = \Res_v e_{v, 1}(\rho_{\Res(\scrQ) \phi}(\phi_\wv)) + (1 + \epsilon \Res(\scrQ) \Res_v \lambda_v t_v(\tau)) \Res_v e_{ v, 2}(\rho_{\Res(\scrQ) \phi}(\phi_\wv)) 
\end{multline*}
for all $v \in Q$, $\tau \in I_{F_\wv}$. This follows on multiplying both sides of the equality $\Res_v = e_{v, 1}(\rho_{\Res(\scrQ) \phi}(\phi_\wv)) + e_{v, 2}(\rho_{\Res(\scrQ) \phi}(\phi_\wv))$ by $\Res_v ( 1 + \epsilon \Res(\scrQ) \phi(\tau) )$ and re-arranging.

We have defined a map $H^1_{\cL(\scrQ)}(F^+, W_m) \to \Hom_\cO(P(\scrQ) / P(\scrQ)^2, \cO  / \varpi^m  \cO)$. It is easy to see that it is in fact a homomorphism of $\cO$-modules. We need to bound the exponent of the kernel and cokernel of this homomorphism. It is helpful here to introduce the commutative diagram
\[ \xymatrix{ H^1_{\cL(\scrQ)}(F^+, W_m) \ar[r] \ar[d] &  \Hom_\cO(P(\scrQ) / P(\scrQ)^2, \cO  / \varpi^m  \cO) \ar[d] \\ 
H^1_{\cL_{S \cup Q}}(F^+, W_m) \ar[r] & \Hom_\cO(P_Q / P_Q^2, \cO / \varpi^m \cO), } \]
where $P_Q \subset R_{S \cup Q}$ is the kernel of the homomorphism $R_{S \cup Q} \to  R_S \to \cO$ associated to $\mathcal{DET}(\rho)$, and the arrows may be described as follows: the left vertical arrow is the natural inclusion, the right vertical arrow is pullback along $R_{S \cup Q} \to R(\scrQ)$, and the bottom horizontal sends $[\phi]$ to the classifying map of $\mathcal{DET}(\rho_{\Res(\scrQ) \phi})$. Using \cite[Proposition 2.20]{New19a}, we get the existence of a constant $d \geq 0$ (not depending on $\scrQ$) such that the kernel and cokernel of the bottom horizontal map are annihilated by $\varpi^d \Res(\scrQ)$. After possibly increasing $d$, we can assume as well that $\rho(\cO[G_{F, S}])$ contains $\varpi^d M_n(\cO)$.

We now establish the analogous claim for the upper map. It is immediate that the kernel of the upper map is also annihilated by $\varpi^d \Res(\scrQ)$. To analyse the cokernel, take a homomorphism $P(\scrQ) / P(\scrQ)^2 \to \cO / \varpi^m$ corresponding to a homomorphism $f : R(\scrQ) \to A_m$, and let $D_0$ be the corresponding group determinant. By the cited proposition, there exists $[\phi] \in H^1_{\cL_{S \cup Q}}(F^+, W_m)$  such that $\alpha_{\varpi^d} \circ D_0$ is the group determinant associated to $\rho_\phi$. Using the defining relations (\ref{eqn_trace_relations}) we find that for all $v \in Q$ and $\tau \in I_{F_\wv}$, we have
\begin{multline*} \Res_{v, q}^{n!} \Res_v^2 \rho_{\varpi^d \phi}(\tau) = \Res_{v, q}^{n!} \Res_v^2(1 + \epsilon \varpi^d \phi(\tau)) \\
= \Res_{v, q}^{n!} \Res_v e_{v, 1}(\rho_{\varpi^d \phi}(\phi_\wv)) + \langle \tau \rangle \Res_{v, q}^{n!} \Res_v e_{v, 2}(\rho_{\varpi^d \phi}(\phi_\wv)). 
\end{multline*} 
We can find $\mu_v \in \cO$ such that $\langle \tau \rangle = 1 + \epsilon  \mu_v t_v(\tau)$. The above identity then gives
\[ \epsilon \Res_{v, q}^{n!} \Res_v^2  \varpi^d \phi(\tau) = \epsilon \Res_{v, q}^{n!} \mu_v \Res_v t_v(\tau) e_{v, 2}(\rho_{\varpi^d \phi}(\phi_\wv)) \]
in $A_m$, hence
\[ \Res_{v, q}^{n!} \Res_v^2  \varpi^d \phi(\tau) = \mu_v \Res_{v, q}^{n!} \Res_v t_v(\tau) e_{v, 2}(\rho(\Frob_\wv)) \]
in $\cO / \varpi^m \cO$. It follows that $[ \Res(\scrQ)^2  \Res_q(\scrQ)^{n!}\varpi^d \phi] \in H^1_{\cL(\scrQ)}(F^+, W_m)$, where we define
\[ \Res_q(\scrQ) =  \operatorname{lcm}( \{ f_\scrQ(\Res_{q, v})  \}_{v \in Q}) \in \cO. \]
This element is a pre-image of $\alpha_{\varpi^{2d} \Res(\scrQ)^3 \Res_q(\scrQ)^{n!}} \circ f$. The proof is complete on noting that $\Res(\scrQ) \equiv \Res_q(\scrQ) \text{ mod }\varpi^N$.
\end{proof}
Here is a variant which will be used later to conclude the vanishing of the adjoint Selmer group.
\begin{lemma}\label{lem_adjoint_Selmer}
Suppose given elements $\sigma_1, \dots, \sigma_q \in G_F$ and factorisations $f_i(X) := \det(X - \rho(\sigma_i)) = f_{i, 1}(X) f_{i, 2}(X)$ for $i = 1, \dots, q$, where for each $i$, $f_{i, 1}(X), f_{i, 2}(X) \in \cO[X]$ are monic polynomials with no common roots in $\overline{\bQ}_p$. Let $\cA_0 = \otimes_{i=1}^q \cA$, $\cB_0 = \otimes_{i=1}^q \cB_{\deg f_{i, 1}, f_{i, 2}}$, and let 
$P_0 \subset R_S \otimes_{\cA_0} \cB_0$ be the kernel of the map $R_S \otimes_{\cA_0} \cB_0 \to \cO$ which classifies the group determinant of $\rho$, together with the factorisations $f_i(X) = f_{i, 1}(X) f_{i, 2}(X)$ for $i = 1, \dots, q$. Then there is an isomorphism
\[ H^1_{g, S}(F^+,  W_E) \cong \Hom_\cO(P_0 / P_0^2, E). \]
\end{lemma}
The definition of the group $H^1_{g, S}(F^+,  W_E)$ is recalled in \cite[\S 1]{New19a}. We note that when $\mathrm{WD}(\rho|_{G_{F_\wv}})$ is generic for each $v \in S$, it equals $H^1_f(F^+, W_E)$.
\begin{proof}
Arguing in the same way as in the proof of the previous proposition shows that there is a an isomorphism of $E$-vector spaces
\[ \left( \varprojlim_m H^1_{\cL_S}(F^+, W_m) \right) \otimes_\cO E \to \left( \varprojlim_m \Hom_\cO(P_0 / P_0^2, \cO / \varpi^m \cO) \right) \otimes_\cO E. \]
The left-hand side may be identified with $H^1_{g, S}(F^+, W_E)$, by \cite[Proposition 2.21]{New19a}. The right-hand side may be identified with $\Hom_\cO(P_0 / P_0^2, E)$ ($P_0 / P_0^2$ is a finitely generated $\cO$-module). This completes the proof.
\end{proof}
\begin{lemma}\label{lem_greenberg_wiles}
Suppose that for each place $v \in S$, $\operatorname{WD}(\rho|_{G_{F_\wv}})$ is generic, in the sense of \cite[Definition 1.1]{New19a}. Then there exists $d \geq 0$ with the following property: for each Taylor--Wiles datum $\scrQ$ of level $N \geq 1$ and for each integer $1 \leq m \leq N$, we have
\[ h^1_{\cL(\scrQ)}(F^+, W_m) \leq d + h^1_{\cL(\scrQ)^\perp}(F^+, W_m(1)) + m |Q|. \]
\end{lemma}
\begin{proof}
This is an application of the Greenberg--Wiles formula, compare \cite[Lemma 2.23]{New19a}. The only additional thing to check here is that if $v \in Q$, then $l(\scrQ)_{v, m} - h^0(F^+_v, W_m)$ is bounded above by $m$. Inspecting the definition of $\cL(\scrQ)$, we see that $l(\scrQ)_{v, m} - h^0(F^+_v, W_m)$ equals the length of the $\cO$-submodule of $\Hom_\cO(I_{F_\wv}, W_m^{G_{F_\wv}})$ generated by the homomorphism $\tau \mapsto t_v(\tau) e_{v, 2}(\rho(\Frob_\wv))$, which is certainly bounded above by $m$.
\end{proof}
\begin{cor}\label{cor_generators_for_Selmer}
Suppose that for each place $v \in S$,  $\operatorname{WD}(\rho|_{G_{F_\wv}})$ is generic. Then there exists $d \geq 0$ such that for every $N \geq 1$ and every Taylor--Wiles datum $\scrQ$ of level $N$, there is a map
\[ \cO^{|Q|} \to H^1_{\cL(\scrQ)}(F^+, W_N) \]
with cokernel  of length $\leq  d + h^1_{\cL(\scrQ)^\perp}(F^+, W_N(1))$.
\end{cor}
\begin{proof}
By \cite[Lemma 2.24]{New19a} and Lemma \ref{lem_greenberg_wiles}, it is enough to show there are constants $d_0, d_1 \geq 0$ such that for every $N \geq 1$ and any Taylor--Wiles datum $\scrQ$ of level $N$, we have
\begin{equation}
l(H^1_{\cL(\scrQ)}(F^+, W_N) / (\varpi^m)) \leq h^1_{\cL(\scrQ)}(F^+, W_m) + d_0
\end{equation}
and
\begin{equation}
l(H^1_{\cL(\scrQ)^\perp}(F^+, W_m(1)) \leq l(H^1_{\cL(\scrQ)^\perp}(F^+, W_N(1)) + d_1.
\end{equation}
This follows by the same argument as in the proof of \cite[Corollary 2.25]{New19a}, provided we can show that for each $m \geq 1$ the natural maps $W_m \to W_{m+1}$ (resp. $W_{m+1} \to W_m$) send $\cL(\scrQ)_{v, m}$ into $\cL(\scrQ)_{v, m+1}$ (resp. $\cL(\scrQ)_{v, m+1}$ into $\cL(\scrQ)_{v, m}$). This is clear from the definitions. 
\end{proof}
\begin{lemma}\label{lem_killing_dual_Selmer_group}
Let $q \geq  \operatorname{corank}_\cO H^1(F_S / F^+, W_{E / \cO}(1))$, and suppose that $\rho$ satisfies the following conditions:
\begin{enumerate}
\item There is a place $v \nmid S$ of $F$ such that all of the eigenvalues of  $\rho(\Frob_v)$ are $q_v^{(n-1)}$-Weil numbers. 
\item $\rho|_{G_{F(\zeta_{p^\infty})}}$ is absolutely irreducible and for each $\sigma \in G_{F(\zeta_{p^\infty})}$, the eigenvalues of $\rho(\sigma)$ all lie in $E$.
\end{enumerate}
Then we can find the following data:
\begin{enumerate}
\item An integer $d \geq 1$.
\item Elements $\sigma_1, \dots, \sigma_q \in G_{F(\zeta_{p^\infty})}$, together with factorisations $f_i(X) := \det(X - \rho(\sigma_i)) = f_{i, 1}(X) f_{i, 2}(X)$, where $f_{i, 1}(X), f_{i, 2}(X)$ are monic, coprime polynomials in $\cO[X]$.
\end{enumerate}
These data have the property that for any Taylor--Wiles datum 
\[ \scrQ = (Q, \widetilde{Q}, (f_{v, 1}(X))_{v \in Q}, (f_{v, 2}(X))_{v \in Q}) \]
 of level $N > d$ such that $Q = \{ v_1, \dots, v_q \}$ and $\rho(\Frob_{\wv_i}) \text{ mod } \varpi^N = \rho(\sigma_i) \text{ mod }\varpi^N$ and $f_{v_i, j}(X) \equiv f_{i, j}(X) \text{ mod }\varpi^N$ for each $i = 1, \dots, q$ and $j = 1, 2$, the following conditions are satisfied:
\begin{enumerate}
\item For each $v \in Q$, we have $\ord_\varpi \operatorname{Res}(f_{v, 1}, f_{v, 2}) \leq d$.
\item $h^1_{\cL(\scrQ)^\perp}(F^+, W_N(1)) \leq d$.
\end{enumerate}
\end{lemma}	
\begin{proof}
We first claim that to prove the lemma, it is enough to find elements $\sigma_1, \dots, \sigma_q \in G_{F(\zeta_{p^\infty})}$ with factorisations $f_i(X) = f_{i, 1}(X) f_{i, 2}(X)$ such that the morphism of $\cO$-modules
\[ H^1(F_S / F^+, W_{E / \cO}(1)) \to \oplus_{i=1}^q E / \cO, \]
\[ [ \phi ] \mapsto ( \tr e_{i, 2}(\rho(\sigma_i)) \phi(\sigma_i) )_{i = 1, \dots, q}, \]
has kernel of finite length. Indeed, suppose given elements with this property. Then there exists $d_0 \geq 0$ such that for all $m \geq 1$, the kernel of the map
\[  H^1(F_S / F^+, W_{m}(1)) \to \oplus_{i=1}^q \cO / \varpi^m \]
\[ [ \phi ] \mapsto ( \tr e_{i, 2}(\rho(\sigma_i)) \phi(\sigma_i) )_{i = 1, \dots, q}, \]
has length bounded above by $d_0$. Suppose that $\scrQ$ is a Taylor--Wiles datum such that $Q = \{ v_1, \dots, v_q \}$ and $\rho(\Frob_{\wv_i}) \text{ mod } \varpi^N = \rho(\sigma_i) \text{ mod }\varpi^N$ and $f_{v_i, j}(X) \equiv f_{i, j}(X) \text{ mod }\varpi^N$ for each $i, j$. Then $H^1_{\cL(\scrQ)^\perp}(F^+, W_N(1))$ is identified with the kernel of the above map (for $m = N$) so has length bounded above by $d_0$. The lemma will hold with $d = \max(d_0, \{ \ord_\varpi \operatorname{Res}( f_{i, 1}(X), f_{i, 2}(X) )\}_{i =1, \dots, q})$. 

We now explain how to finds elements $\sigma_1, \dots, \sigma_q$ with these properties. By induction, it is enough to show that for any non-zero homomorphism $\kappa : E / \cO \to H^1(F_S / F^+, W_{E / \cO}(1))$, we can find an element $\sigma_0 \in G_{F(\zeta_{p^\infty})}$ and factorisation $f_0(X) := \det(X - \rho(\sigma_0)) = f_{0, 1}(X) f_{0, 2}(X)$ such that the homomorphism $\kappa _{\sigma_0} : E/\cO \to E / \cO$, $x \mapsto \tr e_{0, 2}(\rho(\sigma_0)) \kappa (x)(\sigma_0)$ is still non-zero. 

Let $F_\infty = F(\zeta_{p^\infty})$, let $L'_\infty / F^+$ be the extension cut out by $W_E(1)$, and let $L_\infty = L'_\infty \cdot F_\infty$. Then \cite[Lemma 6.2]{Kis04a} implies that $H^1(L'_\infty / F^+, W_{E}(1)) = 0$, hence $H^1(L_\infty / F^+, W_E(1)) = 0$, hence $H^1(L_\infty / F^+, W_{E / \cO}(1))$ has finite length and the restriction of $\kappa $ to $G_{L_\infty}$ is non-zero. (The cited result assumes that $W_E(1)|_{G_{F_v}}$ is pure for but finitely many places $v$, but it is enough to assume purity at a single place, as we do here. In our applications of this result, the stronger condition of purity at all but finitely many places is known to hold.)

We can interpret this restriction as a $G_{F^+}$-equivariant homomorphism $K : E / \cO \to H^1(L_\infty, W_{E / \cO}(1))$. Let $M \subset W_{E / \cO}(1)$ be the $\cO$-submodule generated by the elements $K(x)(\sigma)$, $x \in E / \cO$, $\sigma \in G_{L_\infty}$. Then $M$ is a divisible $\cO$-submodule which is invariant under the action of $G_{F_\infty}$, so by Lemma \ref{lem_weakly_adequate} there exists $x \in E / \cO$, $\tau \in G_{L_\infty}$, $\sigma \in G_{F_\infty}$ with eigenvalue $\alpha \in \cO$ such that $\tr e_{\sigma, \alpha}(\rho(\sigma)) K(x)(\tau) \neq 0$. If $\tr e_{\sigma, \alpha}(\rho(\sigma)) K(x)(\sigma) \neq 0$, we're done on taking $\sigma_0 = \sigma$ and $f_{0, 2}(X) = \operatorname{gcd}(f_0(X), (X - \alpha)^n)$. If $\tr e_{2, \alpha}(\rho(\sigma)) K(x)(\sigma) = 0$, we're done on taking $\sigma_0 = \tau \sigma$ and $f_{0, 2}(X) = \operatorname{gcd}(f_0(X), (X - \alpha)^n)$. 
\end{proof}
\begin{proposition}\label{prop_existence_of_TW_primes}
Let $q \geq  \operatorname{corank}_\cO H^1(F_S / F^+, W_{E / \cO}(1))$, and suppose that $\rho$ satisfies the following conditions:
\begin{enumerate}
\item There is a place $v \nmid S$ of $F$ such that all of the eigenvalues of  $\rho(\Frob_v)$ are $q_v^{(n-1)}$-Weil numbers. 
\item $\rho|_{G_{F(\zeta_{p^\infty})}}$ is absolutely irreducible.
\item For each place $v \in S$, $\operatorname{WD}(\rho|_{G_{F_\wv}})$ is generic. 
\end{enumerate}
Then we can find the following data:
	\begin{enumerate}	
	\item An integer $d \geq 1$.	
	\item Elements $\sigma_1, \dots, \sigma_q \in G_{F(\zeta_{p^\infty})}$, together with factorisations $f_i(X) := \det(X - \rho(\sigma_i)) = f_{i, 1}(X) f_{i, 2}(X)$, where $f_{i, 1}(X), f_{i, 2}(X) \in \cO[X]$ are monic coprime polynomials in $\cO[X]$.
\end{enumerate}	
These data have the property that for any Taylor--Wiles datum 	
\[ \scrQ = (Q, \widetilde{Q}, (f_{v, 1}(X))_{v \in Q}, (f_{v, 2}(X))_{v \in Q}) \]	
of level $N > d$ such that $\rho(\Frob_{\wv_i}) \text{ mod } \varpi^N = \rho(\sigma_i) \text{ mod }\varpi^N$, $f_{v_i, 1}(X) \equiv f_{i, 1}(X) \text{ mod }\varpi^N$ and $f_{v_i, 2}(X) \equiv f_{i, 2}(X) \text{ mod }\varpi^N$ for each $i = 1, \dots, q$, the following conditions are satisfied: there is a map 	
\[ \cO \llbracket x_1, \dots, x_q \rrbracket \to R(\scrQ) \]	
such that the images of $x_1, \dots, x_q$ lie in $P(\scrQ)$ and	
\[ P(\scrQ) / (P(\scrQ)^2, x_1, \dots, x_q) \]	
is an $\cO$-module of length $\leq d$. Moreover, we have $\ord_\varpi \Res(f_{v, 1}, f_{v, 2}) \leq d$.	
\end{proposition}	
\begin{proof}	
We choose the data $\sigma_1, \dots, \sigma_q$ and $f_{i, j}(X) \in \cO[X]$ and integer $d \geq 1$ using Lemma \ref{lem_killing_dual_Selmer_group}. Suppose given a Taylor--Wiles  datum  $\scrQ$ satisfying the conditions in the statement of the proposition. By Corollary \ref{cor_generators_for_Selmer} and Lemma \ref{lem_comparing_Selmer_and_tangent_space}, there exists a morphism of $\cO$-modules $\cO^q \to   P(\scrQ) / P(\scrQ)^2 \otimes_\cO \cO / \varpi^N \cO$ with cokernel annihilated by $\varpi^d$. We define the map $\cO \llbracket x_1, \dots, x_q \rrbracket \to R(\scrQ)$ to send $x_1, \dots, x_q$ to arbitrary lifts to $P(\scrQ)$ of the images of the standard basis elements of $\cO^q$. 	

To finish the proof, we need to show that $P(\scrQ) / (P(\scrQ)^2, x_1, \dots, x_q)$ is an $\cO$-module of uniformly bounded length. Since $N > d$, 	
\[ P(\scrQ) / (P(\scrQ)^2, x_1, \dots, x_q) \]	
is annihilated by $\varpi^d$. The desired result will follow therefore if we can show that there is a bound, independent of $\scrQ$, for the number of generators for 	
$P(\scrQ) / P(\scrQ)^2$. As in the proof of \cite[Corollary 2.31]{New19a}, this follows from the corresponding statement for $\ffrm_{R_{S \cup Q}} / \ffrm_{R_{S \cup Q}}^2$.	
\end{proof}

\section{The main theorem}\label{sec_patching}

In this section we prove our main theorem:
\begin{theorem}\label{thm_vanishing_over_CM_field}
Let $F$ be a CM number field, let $n \geq 2$, and let $(\pi, \chi)$ be a regular algebraic, cuspidal, polarized automorphic representation of $\GL_n(\A_F)$. Let $\iota : \overline{\bQ}_p \to \bC$ be an isomorphism, and suppose that $r_{\pi, \iota}|_{G_{F(\zeta_{p^\infty})}}$ is irreducible. Then $H^1_f(F^+,  \ad r_{\pi, \iota}) = 0$.
\end{theorem}
\begin{proof}
Using the same sequence of reductions as in the proof of \cite[Theorem 5.2]{New19a}, we can assume that $\pi$ satisfies the following additional conditions:
\begin{itemize}
\item $\pi$ is conjugate self-dual.
\item $F / F^+$ is everywhere unramified and $[F^+ : \bQ]$ is even.
\item Let $S$ denote the set of finite places of $F$ at which  $\pi$ is ramified, together  with the $p$-adic places of $F$. Then for each $v \in S$, $v$ is split over $F^+$ and $\pi_v$ is Iwahori-spherical.
\end{itemize}
The proof of the theorem  in this special case will be given in the rest of this section, starting in \S \ref{subsec_start_of_the_proof}.
\end{proof}
Theorem  \ref{thm_vanishing_over_CM_field} has the following consequence for automorphic representations over totally real fields. We refer to \cite[\S 5]{New19a} for the definition of the representation $\mathfrak{gs}$ appearing in the statement (it is the Lie algebra of a general similitude group of a $(\pm)$-symmetric bilinear form, whose parity depends on the parity of the polarizing character $\chi$):
\begin{theorem}\label{thm_vanishing_over_real_field}
Let $F$ be a totally real number field,  and let $(\pi, \chi)$ be a regular algebraic, cuspidal, polarized automorphic representation of  $\GL_n(\A_F)$. Let $\iota : \overline{\bQ}_p \to  \bC$ be an isomorphism, and suppose that $r_{\pi, \iota}|_{G_{F(\zeta_{p^\infty})}}$ is  irreducible. Then $H^1_f(F^+, \mathfrak{gs}) = 0$.
\end{theorem}
\begin{proof}
This can be deduced from Theorem \ref{thm_vanishing_over_CM_field} using base change, cf. \cite[Theorem B]{All16}.
\end{proof}

\subsection{Start of the proof}\label{subsec_start_of_the_proof}
 We begin  by repeating, almost verbatim, the set-up from \cite[\S 4]{New19a}; the arguments will diverge when we begin to describe the Hecke algebras associated to Taylor--Wiles data. 

We therefore suppose given $n \geq2$, a CM number field $F$, a cuspidal, regular algebraic, conjugate self-dual automorphic representation $\pi$ of $\GL_n(\bA_F)$, and an isomorphism $\iota : \overline{\bQ}_p \to \bC$. We assume that the following conditions are satisfied: 
\begin{itemize}
\item $F / F^+$ is everywhere unramified and $[F^+ : \bQ]$ is even.
\item Let $S$ denote the set of finite places of $F^+$ above which $\pi$ is ramified, together with the $p$-adic places of $F^+$. Then for each $v \in S$, $v$ splits $v  = w w^c$ in $F$ and $\pi_w$ is Iwahori-spherical.
\item  $r_{\pi, \iota}|_{G_{F(\zeta_{p^\infty})}}$ is absolutely irreducible.
\end{itemize}
We remark that for each place $v$ of $F$, $\operatorname{WD}(r_{\pi, \iota}|_{G_{F_v}})$ is generic \cite{Caraianilnotp, Car14}. We choose an extension of $r_{\pi, \iota}$ to a homomorphism $G_{F^+} \to \cG_n(\overline{\bQ}_p)$, which then gives the action of $G_{F^+}$ on $\ad r_{\pi, \iota}$. We have $\nu \circ r_{\pi, \iota} = \delta_{F / F^+}^n \epsilon^{1-n}$. We must show that $H^1_f(F^+, \ad r_{\pi, 
	\iota}) = 0$.

We can find the following data:
\begin{itemize}
	\item For each place $v \in S$, a choice of place $\wv$ of $F$ lying above $v$. We set $\widetilde{S} = \{ \wv \mid v \in S \}$ and $\widetilde{S}_p = \{ \wv \mid v \in S_p \}$.
	\item A Hermitian form $\langle \cdot, \cdot \rangle : F^n \times F^n \to F$ such that the associated unitary group $G$ (defined on $R$-points by $G(R) = \{ g \in \GL_n(F \otimes_{F^+} R) \mid g^\ast g = 1 \}$) is definite at infinity and quasi-split at each finite place of $F^+$.
	\item A reductive group scheme over $\cO_{F^+}$ extending $G$ (also denoted $G$).
	\item For each finite place $v = w w^c$ of $F^+$ which splits in $F$, an 
	isomorphism $\iota_w : G_{\cO_{F^+_v}} \to \Res_{\cO_{F_w} / \cO_{F^+_v}} 
	\GL_n $ of group schemes over $\cO_{F^+_v}$. We assume that the 
	induced isomorphism $\iota_w : G(F^+_v) \to \GL_n(F_w)$ is in the same 
	inner class as the isomorphism given by inclusion $G(F_v^+) \subset 
	\GL_n(F_w) \times \GL_n(F_{w^c})$, followed by projection to the first 
	factor.
	\item An automorphic representation $\sigma$ of $G(\bA_{F^+})$ with the following properties:
	\begin{itemize}
	\item For each finite place $v$ of $F^+$ which is inert in $F$, $\sigma_v^{G(\cO_{F^+_v})} \neq 0$ and $\sigma_v$, $\pi_v$ are related by unramified base change.	
	\item For each finite place $v$ of $F^+$ which is split $v = w w^c$ in $F$, $\sigma_v \cong \pi_w 	
	\circ \iota_w$.
		\item If $v | \infty$ is a place of $F^+$, then the infinitesimal character of $\sigma_v$ respects that of $\pi_v$ under base change. 
	\end{itemize}
	\item An open compact subgroup $U = \prod_v U_v$ of $G(\bA_{F^+}^\infty)$ with the following properties:
	\begin{itemize}
		\item For each place $v \in S_p$, $U_v = \iota_{\wv}^{-1}(\Iw_{\wv})$, where $\Iw_\wv \subset \GL_n(\cO_{F_\wv})$ is the standard Iwahori subgroup.
		\item For each inert place $v$ of $F^+$, $U_v = G(\cO_{F^+_v})$.
		\item $(\sigma^\infty)^U \neq 0$.
		\item $U$ is sufficiently small: for all $g \in G(\bA_{F^+}^\infty)$, $g U g^{-1} \cap G(F^+) = \{ 1 \}$.
	\end{itemize}
\end{itemize}
(We can find such a $G$ because $[F^+ : \bQ]$ is even. The existence of $\sigma$ is deduced from that of $\pi$ using \cite[\S 5]{labesse}.)
We can regard $\sigma_\infty$ as an algebraic representation of the group 
$(\Res_{F^+ / \bQ} G)_\bC$. Let $\widetilde{I}_p \subset \Hom(F, 
\overline{\bQ}_p)$ denote the set of 
embeddings inducing places $\wv \in \widetilde{S}_p$. Then our choices 
determine an isomorphism
\[ (\Res_{F^+ / \bQ} G)_{\overline{\bQ}_p} \cong \prod_{\tau \in \widetilde{I}_p} \GL_n. \]
Let $\lambda = (\lambda_\tau)_{\tau \in \widetilde{I}_p} \in 
(\bZ_+^n)^{\widetilde{I}_p}$ denote the highest weight of the algebraic 
representation $V_\lambda$ of $(\Res_{F^+ / \bQ} G)_{\overline{\bQ}_p}$ such 
that $V_\lambda \otimes_{\iota, \overline{\bQ}_p} \bC \cong 
\sigma_\infty^\vee$. We can define a highest weight $\xi$ for 
$(\Res_{F/\bQ}\GL_n)_{\overline{\bQ}_p}$ by letting $\xi_\tau = \lambda_\tau$ 
and $\xi_{\tau c} = -w_0\lambda_\tau$ for $\tau \in \widetilde{I}_p$ ($w_0$ is 
the longest element in the Weyl group of $\GL_n$). The 
infinitesimal character of 
$\pi_\infty$ is the same as that of $V_\xi^\vee\otimes_{\iota,\Qpbar}\CC$.   We fix once and for all integers $a \leq b$ such that for all $\tau \in \Hom(F, \overline{\bQ}_p)$, the elements of $\mathrm{HT}_\tau(r_{\pi, \iota})$ are contained in $[a, b]$ and $a + b = n-1$. 

We can find a representation $\cV_\lambda$ of the group scheme 
$(\Res_{\cO_{F^+} / \bZ} G)_\cO$, finite free over $\cO$, and such that $\cV_\lambda \otimes_\cO 
\overline{\bQ}_p \cong V_\lambda$. Thus $\cV_\lambda(\cO)$ is a finite free $\cO$-module 
which receives an action of $U_p = \prod_{v \in S_p} U_v$. For any open compact 
subgroup $V = \prod_v V_v \subset U$, and any $\cO$-algebra $A$, we define 
$S_\lambda(V, A)$ to be the set of functions $f : G(\bA_{F^+}^\infty) \to 
\cV_\lambda(A)$ such that for each $v \in V$, $\gamma \in G(F^+)$, $g \in 
G(\bA_{F^+}^\infty)$, $v_p f(\gamma g v) = f(g)$. We observe that
\[ \varinjlim_{U^p} S_\lambda(U^p U_p, A) \]
has a natural structure of $A[U^p]$-module, and the $U^p$-invariants are $S_\lambda(U, A)$. It follows that $S_\lambda(U, A)$ has a natural structure of $\cH(G(\bA_{F^+}^{\infty, p}), U^p)$-module. There is an isomorphism of $\cH(G(\bA_{F^+}^{\infty, p}), U^p)$-modules
\[ S_\lambda(U, \cO) \otimes_{\iota, \cO} \bC \cong \oplus_{\mu} (\mu^\infty)^U, \]
where the sum is over automorphic representations of $G(\bA_{F^+})$ (with multiplicity) such that $\mu_\infty \cong \sigma_\infty$. 

Let $E / \bQ_p$ be a coefficient field containing the image of every embedding $F \to \overline{\bQ}_p$. After possibly enlarging $E$, we can assume that there is a model $\rho : G_{F, S} \to \GL_n(\cO)$ of $r_{\pi, \iota}$, which extends to a homomorphism $r : G_{F^+, S} \to \cG_n(\cO)$ such that $\nu \circ r = \epsilon^{1-n} \delta_{F / F^+}^n$. We moreover assume that $E$ contains every quadratic extension of $\bQ_p$, and that for each $\sigma \in G_{F, S}$, the characteristic polynomial $\det(X - \rho(\sigma))$ splits into linear factors in $\cO[X]$.	

Let $\overline{D} = \mathcal{DET}(\overline{\rho})$. With these choices the pseudodeformation ring denoted $R_S = R(\emptyset)$ in \S \ref{sec_deformation} is defined, as well as the prime ideal $P(\emptyset)
= \ker(R(\emptyset) \to \cO)$ determined by $\rho$. Moreover, for any Taylor--Wiles 
datum $\scrQ = (Q, \widetilde{Q}, (f_{v, 1}(X))_{v \in Q}, (f_{v, 2}(X))_{v \in Q})$
we have the auxiliary ring $R(\scrQ)$ and prime ideal $P(\scrQ) = \ker(R(\scrQ) \to \cO)$.

If $V = \prod_v V_v$ is an open compact subgroup of $U$ and $T$ is a finite set 
of places of $F^+$ containing all places such that $V_v \neq G(\cO_{F^+_v})$, 
then we write $\bT_\lambda^T(V, A)$ for the $A$-subalgebra of 
$\End_A(S_\lambda(V, A))$ generated by the unramified Hecke operators at 
split places away from $T$. The existence of 
$\sigma$ implies the existence of a homomorphism
\[ h_{V, \sigma} : \bT^T_\lambda(V, \cO) \to \cO \]
giving the Hecke eigenvalues of $\iota^{-1} \sigma^\infty$. On the other hand, the results of \cite[\S 5]{labesse} (base change), together with the existence of Galois representations associated to cuspidal, polarizable, regular algebraic automorphic representations of $\GL_n(\A_F)$,  imply the existence of a group determinant $D_{V, \lambda}$ of $G_{F}$ valued in $\bT^T_{\lambda}(V, \cO)$ (construction as in \cite[Proposition 4.11]{jackreducible}).

Let $\ffrm \subset \bT^S_\lambda(U, \cO)$ denote the unique maximal ideal containing $\ker h_{U, \sigma}$, and set
\[ S_\emptyset = S_\lambda(U, \cO)_\ffrm, \bT_\emptyset = \bT^S_\lambda(U, \cO)_\ffrm. \]
Then (\cite[Lemma 5.4]{New19a}) there is a surjective homomorphism $R(\emptyset) \to \bT_\emptyset$ classifying the image of $D_{U, \lambda}$ over $\bT_\emptyset$.

Now suppose that $\scrQ = (Q, \widetilde{Q}, (f_{v, 1}(X))_{v \in Q}, (f_{v, 2}(X))_{v \in Q})$ is a Taylor--Wiles datum. If $v \in Q$, we write $\mathfrak{p}_v \subset \GL_n(\cO_{F_\wv})$ for the standard parahoric subgroup associated to the partition $n = \deg f_{v, 1} + \deg f_{v, 2}$ and $\mathfrak{p}_{v, 1} \subset \mathfrak{p}_v$ for the kernel of the associated map $\mathfrak{p}_v \to \GL_{\deg f_{v, 2}}(k(\wv)) \to k(\wv)^\times(p) =\Delta_v$  (notation as in \S \ref{sec_hecke}). We define open compact subgroups $U_1(\scrQ) \subset U_0(\scrQ) \subset U$ as follows: $U_0(\scrQ) = \prod_v U_0(\scrQ)_v$ and $U_1(\scrQ) = \prod_v U_1(\scrQ)_v$, where $U_0(\scrQ)_v = U_1(\scrQ)_v = U_v$ if $v \not\in Q$ and $U_0(\scrQ)_v = \iota_\wv^{-1} \mathfrak{p}_v$ and $U_1(\scrQ)_v = \iota_\wv^{-1} \mathfrak{p}_{v, 1}$ if $v \in Q$.

Thus there is a canonical isomorphism $U_0(\scrQ) / U_1(\scrQ) \cong \Delta_Q = \prod_{v \in Q} \Delta_v$. We make $S_\lambda(U_0(\scrQ), \cO)$ into a  $\cB(\scrQ)$ by twisting the action of each algebra $\cB_{\deg f_{v, 1}, \deg f_{v,2 }}$ ($v \in Q$) by the character $| \cdot |^{(1-n)/2}$ (this twist is necessary to accord with the statement of Corollary \ref{cor_trace_relation_1}). We define $\bT^{S, \scrQ}_\lambda(U_0(\scrQ), \cO)$ to be the commutative $\cO$-subalgebra of $\End_\cO(S_\lambda(U_0(\scrQ), \cO)$ generated by the unramified Hecke operators at split places $v\not\in S \cup Q$, together with the image of the ring $\cB(\scrQ)$. Thus $\bT_\lambda^{S \cup Q}(U_0(\scrQ), \cO) \subset \bT_\lambda^{S, \scrQ}(U_0(\scrQ), \cO)$. We define $\ffrm_{0, \scrQ}$ to be the pullback of $\ffrm$ to $\bT_\lambda^{S \cup Q}(U_0(\scrQ), \cO)$ and define 
\[ S_{\scrQ, 0} = S_\lambda(U_0(\scrQ), \cO)_{\ffrm_{0, \scrQ}}, \bT_{\scrQ, 0} = \bT_\lambda^{S, \scrQ}(U_0(\scrQ), \cO)_{\ffrm_{0, \scrQ}}. \]
\begin{lemma}\label{lem_map_R_to_T_at_U_0_Q_level}
There is a canonical surjective homomorphism $R(\scrQ) \to \bT_{\scrQ, 0}$. 
\end{lemma}
\begin{proof}
We recall that $R(\scrQ)$ is a quotient of $R_{S \cup Q} \otimes_{\cA(\scrQ)} \cB(\scrQ) \otimes_\cO \cO[\Delta_Q]$. The map we construct will factor through the quotient $R_{S \cup Q} \otimes_{\cA(\scrQ)} \cB(\scrQ)$. There is a surjective map $R_{S \cup Q} \otimes_\cO \cB(\scrQ) \to \bT_{\scrQ, 0}$ coming from the group determinant $D_{U_0(\scrQ), \lambda}$ and the canonical map $\cB(\scrQ) \to \bT_{\scrQ, 0}$; therefore what we need to check is first that this map factors through the quotient 
\[ R_{S \cup Q} \otimes_\cO \cB(\scrQ) \to R_{S \cup Q} \otimes_{\cA(\scrQ)} \cB(\scrQ) \] and second that for all $v \in Q$, $\tau \in I_{F_\wv}$, $\sigma \in G_{F, S \cup Q}$, its kernel contains the element
\[ \Res_{v, q}^{n!} \Lambda_1(\sigma(\Res_v^2 \tau - \Res_v e_{v, 1}(\phi_\wv) - \langle \tau \rangle \Res_v e_{v, 2}(\phi_\wv))) \]
appearing in (\ref{eqn_trace_relations}). The first claim is equivalent to the assertion that the two actions of $\cA(\scrQ)$ on $S_{\scrQ, 0}$ induced by the $R_{S \cup Q}$- and $\cB(\scrQ)$-module structures agree. This follows from Corollary \ref{cor_centre_of_parahoric_hecke_algebra}. For the second claim, it is enough to show that for each automorphic representation $\mu$ of $G(\A_{F^+})$ such that $\iota^{-1} \mu^\infty$ contributes to $S_{\scrQ, 0}$, we have the relation
\[ \Res_{v, q}^{n!} \Res_v^2 r_{\mu , \iota}(\tau) = \Res_{v, q}^{n!} \Res_v^2  \]
in $M_n(\cB_{\iota^{-1} \mu_v})$, where $\cB_{\iota^{-1} \mu_v}$ is the $\overline{\bQ}_p$-subalgebra of  $\End_{\overline{\bQ}_p}(\iota^{-1} \mu_v^{U_0(\scrQ)_v})$ generated by the image of $\cB_{\deg f_{v, 1},\deg f_{v, 2}}$. This follows from Corollary \ref{cor_trace_relation_0}: either $\Res_{v, q}^{n!} \iota^{-1} \mu_v^{U_0(\scrQ)_v} = 0$, in which case both sides are zero, or $\mu_v$ is unramified, in which case $r_{\mu, \iota}(\tau) = 1$.
\end{proof}
The space $S_\lambda(U_1(\scrQ), \cO)$ has a canonical structure of module over the ring $\cB(\scrQ)[\Delta_Q]$. Moreover,  $S_\lambda(U_1(\scrQ), \cO)$ is free over $\cO[\Delta_Q]$ and the trace map induces an isomorphism  $S_\lambda(U_1(\scrQ), \cO) \otimes_{\cO[\Delta_Q]} \cO \cong S_\lambda(U_0(\scrQ), \cO)$ (\cite[Lemma 4.6]{New19a}). We define $\bT^{S, \scrQ}_\lambda(U_1(\scrQ), \cO)$ to be the $\cO$-subalgebra of $\End_\cO(S_\lambda(U_1(\scrQ), \cO))$ generated by the unramified Hecke operators at split places $v\not\in S \cup Q$, together with the image of the ring $\cB(\scrQ)[\Delta_Q]$. Thus $\bT_\lambda^{S \cup Q}(U_1(\scrQ), \cO) \subset \bT_\lambda^{S, \scrQ}(U_1(\scrQ), \cO)$. We define $\ffrm_{1, \scrQ}$ to be the pullback of $\ffrm$ to $\bT_\lambda^{S \cup Q}(U_1(\scrQ), \cO)$ and define 
\[ S_{\scrQ, 1} = S_\lambda(U_1(\scrQ), \cO)_{\ffrm_{1, \scrQ}}, \bT_{\scrQ, 1} = \bT_\lambda^{S, \scrQ}(U_1(\scrQ), \cO)_{\ffrm_{1, \scrQ}}. \]
Thus there is a canonical surjective homomorphism $\bT_{\scrQ, 1} \to \bT_{\scrQ, 0}$ (here we apply Lemma \ref{lem_restriction_to_P_invariants}).
\begin{lemma}\label{lem_map_R_to_T_at_U_1_Q_level}
The homomorphism $R(\scrQ) \to \bT_{\scrQ, 0}$ lifts to a surjective $\cO[\Delta_Q]$-algebra homomorphism $R(\scrQ) \to \bT_{\scrQ, 1}$. 
\end{lemma}
\begin{proof}
The proof is similar to the proof of Lemma \ref{lem_map_R_to_T_at_U_0_Q_level}. There is a map $R_{S \cup Q} \otimes_{\cO} \cB(\scrQ)[\Delta_Q] \to \bT_{\scrQ, 1}$ arising by tensor product of the maps $R_{S \cup Q} \to  \bT_{\scrQ, 1}$ classifying $h_{U_1(\scrQ), \lambda}$ and the canonical map $\cB(\scrQ) \otimes_\cO \cO[\Delta_Q] \to \bT_{\scrQ, 1}$. This map factors through $R_{S \cup Q} \otimes_{\cA(\scrQ)} \cB(\scrQ) \otimes_\cO \cO[\Delta_Q]$ by Proposition \ref{prop_trace_relation_1}. To complete the proof, we need to show the kernel of the resulting map contains the elements 
\[  \Res_{v, q}^{n!} \Lambda_1(\sigma(\Res_v^2 \tau - \Res_v e_{v, 1}(\phi_\wv) - \langle \tau \rangle \Res_v e_{v, 2}(\phi_\wv))), \]
or even that for each automorphic representation $\mu$ of $G(\A_{F^+})$ such that $\iota^{-1} \mu^\infty$ contributes to $S_{\scrQ, 1}$, we have the relations
\[ \Res_{v, q}^{n!} \Res_v^2 r_{\mu, \iota}(\tau) = \Res_{v, q}^{n!} \left( \Res_v e_{v, 1}(r_{\mu, \iota}(\phi_\wv)) + \langle \tau \rangle \Res_v e_{v, 2}(r_{\mu, \iota}(\phi_\wv)) \right) \]
in $M_n(\cB_{\iota^{-1}\mu_v, 1})$, where $\cB_{\iota^{-1} \mu_v, 1}$ is the $\overline{\bQ}_p$-subalgebra of  $\End_{\overline{\bQ}_p}(\iota^{-1} \mu_v^{U_1(\scrQ)_v})$ generated by the image of $\cB_{\deg f_{v, 1},\deg f_{v, 2}}[\Delta_v]$.
This follows from Corollary \ref{cor_trace_relation_0} and  Corollary \ref{cor_trace_relation_1}.
\end{proof}
	We need to control the difference between $S_\emptyset$ and $S_{\scrQ, 0}$. There is a homomorphism of $R_{S \cup Q} \otimes_{\cA(\scrQ)} \cB(\scrQ)$-modules:	
\begin{align*} f_{\scrQ} : \cB(\scrQ) \otimes_{\cA(\scrQ)}  S_\emptyset & \to S_{\scrQ, 0}  \\	
	s \otimes x & \mapsto sx	
\end{align*}	
(see Lemma \ref{lem_comparison_map_independent_of_choices}). The following result will be used to control the kernel and cokernel of $f_{\scrQ, m} = f_{\scrQ} \otimes_\cO \cO / \varpi^m \cO$ (when $\scrQ$ has level $N$ and $1 \leq m \leq N$):	
\begin{proposition}	
	Suppose that $\scrQ$ has level $N$ and that $1 \leq m \leq N$. Then the element $\prod_{v \in Q} \Res_v \in \cB(\scrQ)$ annihilates the kernel and cokernel of each of the maps $f_{\scrQ, m}$.
\end{proposition}	
\begin{proof}	
	By Proposition \ref{prop_quasi_inverse}, there is a morphism $g_{\scrQ} : S_{\scrQ, 0}  \to \cB(\scrQ) \otimes_{\cA(\scrQ)} S_\emptyset $ of $R_{S \cup Q} \otimes_{\cA(\scrQ)} \cB(\scrQ)$-modules such that $f_{\scrQ} g_{\scrQ} \text{ mod }\varpi^m$ and $g_{\scrQ} f_{\scrQ} \text{ mod }\varpi^m$ are both given by multiplication by $\prod_{v \in Q} \Res_v$. This implies the desired result. (Here we use that $U$ is sufficiently small, cf. the discussion on \cite[p. 98]{cht}, so that e.g.\ $S_{\scrQ, 0} / (\varpi^m)$ may be viewed as the space of $\prod_{v \in Q} \p_v$-invariants in a suitable $\cO / (\varpi^m)[\prod_{v \in Q} \GL_n(F_\wv)]$-module.)	
\end{proof}

\subsection{Patching}	

We now collect together the data necessary to carry out the patching argument. We argue using ultrafilters, following Pan \cite{Lue} in a similar way to \cite{New19a}.	
\begin{itemize}	
	\item First fix $q = \operatorname{corank}_\cO H^1(F_S / F^+, \ad r(1) \otimes_\cO E /  \cO)$. Let $R_\infty = \cO \llbracket x_1, \dots, x_q \rrbracket$.	
	\item We next fix $d \geq 1$, elements $\sigma_1, \dots, \sigma_q \in G_{F(\zeta_{p^\infty})}$, and factorisations $\det(X - \rho(\sigma_i)) = f_{i, 1}(X) f_{i, 2}(X)$ satisfying the conclusion of  Proposition \ref{prop_existence_of_TW_primes}. For each $N > d$, we can find a Taylor--Wiles datum 	
	\[ \scrQ_N = (Q_N, \widetilde{Q}_N, (f_{v, 1}(X))_{v \in Q}, (f_{v, 2}(X))_{v \in Q}), \] where $Q_N = \{ v_{N, 1}, \dots, v_{N, q} \},$  and the following additional conditions are satisfied:	
	\begin{itemize}	
		\item The characteristic polynomials of $\Frob_{\wv_{N, i}}$ and $\sigma_i$ over $R(\emptyset) / \ffrm_{R(\emptyset)}^N$ agree.	
		\item The characteristic polynomials of $\Frob_{\wv_{N, i}}$ and $\sigma_i$ over $\bT_{\emptyset} / (\varpi^N)$ agree.	
		\item $\rho(\Frob_{\wv_{N, i}}) \equiv \rho(\sigma_i) \text{ mod }\varpi^N$ for each $i = 1, \dots, q$.	
		\item For each $i = 1, \dots, q$, we have $f_{i, 1}(X) \equiv f_{v_{N, i}, 1}(X) \text{ mod }\varpi^N$ and $f_{i, 2}(X) \equiv f_{v_{N, i}, 2}(X) \text{ mod }\varpi^N$. Moreover, $\ord_\varpi \Res(f_{v_{N, i}, 1}, f_{v_{N, i}, 2}) \leq d$.	
	\end{itemize} 	
	(This is possible by the Chebotarev density theorem and Hensel's lemma.) We write $R_N = R(\scrQ_N)$ and $\cA_N = \cA(\scrQ_N)$, $\cB_N = \cB(\scrQ_N)$. We write $P_N = P(\scrQ_N) = \ker(R_N \to \cO)$.	
	\item We set $S_\infty = \cO \llbracket \bZ_p^q \rrbracket$ and fix for each $N \geq 1$ a surjection $\bZ_p^q \to \Delta_{Q_N}$. This gives each ring $R(\scrQ_N)$ the structure of $S_\infty$-algebra. We write $\mathbf{a}_\infty \subset S_\infty$ for the augmentation ideal. We also set $\cA_0 = \otimes_{i=1}^q \bZ[e_1, \dots, e_n]$ and $\cB_0 = \otimes_{i=1}^q \bZ[a_1, \dots, a_{\deg f_{i, 1}}, b_{1}, \dots, b_{\deg f_{v, 2}}]$. The choice of elements $\sigma_1, \dots, \sigma_q$ gives $R(\emptyset)$ the structure of $\cA_0$-algebra. We define $R_0 = R(\emptyset) \otimes_{\cA_0} \cB_0$. There are isomorphisms $\cA_0 \cong \cA_N$ and $\cB_0 \cong \cB_N$ for any $N \geq 1$. We define $P_0 \subset \cA_0$ to be the kernel of the map $R_0 \to \cO$ associated to $\mathcal{DET}(\rho)$ and factorisations $f_i(X) = f_{i, 1}(X) f_{i, 2}(X)$ ($i = 1, \dots, q$).	
	\item Finally, we fix a non-principal ultrafilter $\cF$ on $\{ N \in \bN \mid N > d \}$, and set $\bR = \prod_{N > d} \cO$. If $I \in \cF$, then we define $e_I = (\delta_{N \in I})_{N > d} \in \bR$. Then $e_I$ is an idempotent and $S = \{ e_I \mid I \in \cF \}$ is a multiplicative subset of $\bR$, and we define $\bR_\cF = S^{-1} \bR$. Note that the map $\bR \to \bR_\cF$ factors through $\prod_{N \geq m} \cO$ for any $m > d$.	
\end{itemize}	
We remark that if $N > d$ then there is no canonical map $R_N \to R_0$, but our choice of Taylor--Wiles data means that the map $R_{S \cup Q_N} \otimes_\cO \cB_N \to R_S \otimes_\cO \cB_0$ descends to a surjection $R_N \to R_0 / \ffrm_{R(\emptyset)}^N$. Similarly, there is a canonical surjection $R_N \to \bT_\emptyset / (\varpi^N) \otimes_{\cA_0} \cB_0$.	
We define modules:	
\begin{itemize}	
	\item $M_1 = \varprojlim_m \bR_\cF \otimes_\bR \prod_{N \geq m} S_{\scrQ, 1} / (\mathfrak{m}_{S_\infty}^m)$.	
	\item $M_0 = \varprojlim_m \bR_\cF \otimes_\bR \prod_{N \geq m} S_{\scrQ, 0} /(\varpi^m)$.	
	\item $M = \varprojlim_m \bR_\cF \otimes_\bR \prod_{N \geq m} S_{\emptyset} \otimes_{\cA_N} \cB_N / (\varpi^m)$.	
\end{itemize}	
When $N \geq m$, the $\cA_N$-- and $\cA_0$--actions on $S_\emptyset / (\varpi^m)$ are the same. Thus there is a natural isomorphism $M \cong S_\emptyset \otimes_{\cA_0} \cB_0$, where $\cA_0$ acts via the canonical map $\cA_0 \to R_\emptyset \to \bT_\emptyset$.	
\begin{lemma}\label{lem_comparison_maps_for_patched_mdoules}	
	\begin{enumerate}	
		\item $M_1$ is a flat $S_\infty$-module.	
		\item The trace maps induce an isomorphism $M_1 \otimes_{S_\infty} \cO \cong M_0$. 	
		\item The maps $f_{\scrQ_N, m}$ induce a map $f : M \to M_0$, with kernel and cokernel annihilated by $(\prod_{v \in Q_N} \Res_v^2)_{N > d} \in \prod_{N > d} R(\scrQ_N)$.	
	\end{enumerate}	
\end{lemma}	
\begin{proof}	
	The proof is the same as the proof of \cite[Lemma 4.13]{New19a}.	
\end{proof}	
We define	
\[ R^\prm = \varprojlim_m \bR_\cF \otimes_\bR \prod_{N > d} R_N / (\ffrm_{R_{S \cup Q_N}} \prod_{v \in Q_N} \Res_v^2)^m. \]	
Then $M_1$, $M_0$, and $M$ have natural structures of $R^\prm$-modules with respect to which the maps of Lemma \ref{lem_comparison_maps_for_patched_mdoules} are morphisms of $R^\prm$-modules (same proof as \cite[Lemma 4.15]{New19a}). There is also a natural map	
\[ R^\prm \to \varprojlim_m \bR_\cF \otimes_\bR \prod_{N > d} R_0 / \ffrm_{R(\emptyset)}^m\cong R_0. \]	
\begin{lemma}	
	\begin{enumerate} \item The map $R^\prm \to  R_0$ just defined is surjective.  The action of $R^\prm$ on $M$ factors through this map.	
		\item Let $P^p$ denote the pre-image of $P_0$ under this map. Then $P^p$ equals the image of $\prod_{N > d} P_N \subset \prod_{N > d} R_N$ under the map $\prod_{N > d} R_N\to R^\prm$.	
		\item For each $k \geq 1$, the ideal $(P^\prm)^k$ equals the image of $\prod_{N> d} P_N^k \subset \prod_{N > d} R_N$ in $R^\prm$.	
	\end{enumerate}	
\end{lemma}	
\begin{proof}	
	The first part is proved in the same way as \cite[Lemma 4.17]{New19a}. The second and third parts can be proved in the same way as \cite[Lemma 4.19, Lemma 4.20]{New19a}.	
\end{proof}	
For each $N > d$, Proposition \ref{prop_existence_of_TW_primes} implies the existence of a map 
\[ R_\infty = \cO \llbracket x_1, \dots, x_q \rrbracket \to R_N \]
 which sends $x_1, \dots, x_q$ into $P_N$, and such that $P_N / (P_N^2, x_1, \dots, x_q)$ has uniformly bounded length (as $\cO$-module). There is an induced map $R_\infty \to R^\prm$ which sends the ideal $P_\infty = (x_1, \dots, x_q)$ into $P^\prm$.	
\begin{proposition}\label{prop_completed_patched_def_ring}	
	The natural map $R_\infty \to R^\prm$ induces a surjection $(R_\infty)^{\widehat{}}_{P_\infty} \to (R^\prm)^{\widehat{}}_{P^\prm}$ on completed local rings. In particular, $(R^\prm)^{\widehat{}}_{P^\prm} \in \cC_E$.	
\end{proposition}	
\begin{proof}	
	The proof is the same as the proof of \cite[Proposition 4.22]{New19a}.	
\end{proof}	
We next define quotients of our patched modules as follows:	
\begin{itemize}	
	\item $\mrm_1 = (M_1 / \mathbf{a}_\infty^2)_{P^\prm}$.	
	\item $\mrm_0 = (M_0)_{P^\prm}$.	
	\item $\mrm = M_{P^p} = M_{P_0}$.	
\end{itemize}	
\begin{proposition}\label{prop_S_infty_freeness}	
	\begin{enumerate}	
		\item The map $f : M \to M_0$ induces an isomorphism $\mrm \to \mrm_0$.	
		\item The trace maps induce an isomorphism $\mrm_1 / (\mathbf{a}_\infty) \cong \mrm_0$.	
		\item $\mrm_1$ is a finite free $S_{\infty, \mathbf{a}_\infty} / (\mathbf{a}_\infty^2)$-module.	
	\end{enumerate}	
\end{proposition}	
\begin{proof}	
	The first part is true because the image of the element $(\prod_{v \in Q_N} \Res_v^2)_{N > d} \in \prod_{N > d} R_N$ in $R^\prm$ is not in $P^\prm$. The second part is true because the analogous statement is true before localization. The third part is true because $\mrm_1$ is both flat (as $M_1$ is flat) and finitely generated (because $S_{\infty, \mathbf{a}_\infty} / (\mathbf{a}_\infty^2)$ is Artinian and $\mrm_0$ is a finite-dimensional $E$-vector space).	
\end{proof}	
Finally, we conclude:	
\begin{proposition}	
	$\mrm$ is a free $(R_0)^{\widehat{}}_{P_0}$-module. Consequently, $H^1_f(F^+, \ad r_{\pi, \iota}) = 0$.	
\end{proposition}	
\begin{proof}	
	According to Proposition \ref{prop_completed_patched_def_ring}, $(R^\prm)^{\widehat{}}_{P^\prm}$ is a quotient of $(R_\infty)^{\widehat{}}_{P_\infty}$, which is a complete Noetherian regular local ring of dimension $q$. Applying Brochard's criterion \cite[Theorem 1.1]{brochard} (along with the third part of Proposition \ref{prop_S_infty_freeness}), we conclude that  $\mrm_1$ is a free $(R^\prm)^{\widehat{}}_{P^\prm} / (\mathbf{a}_\infty^2)$-module, and hence that $\mrm_0 \cong \mrm$ is a free $(R^\prm)^{\widehat{}}_{P^\prm} / (\mathbf{a}_\infty)$-module. Since the action of $(R^\prm)^{\widehat{}}_{P^\prm} / (\mathbf{a}_\infty)$ on $\mrm$ factors through the map $(R^\prm)^{\widehat{}}_{P^\prm} \to (R_0)^{\widehat{}}_{P_0}$, it must be the case that $\mrm$ is a free $(R_0)^{\widehat{}}_{P_0}$-module. To prove the vanishing of $H^1_f(F^+, \ad r_{\pi, \iota})$ and finish the proof, we need to check the following two points: 	
	\begin{itemize}	
		\item $\mrm$ is a semisimple $(R_0)^{\widehat{}}_{P_0}$-module.	
		\item The $E$-vector spaces $H^1_f(F^+, \ad r \otimes_\cO E)$ and $P_0 / P_0^2 \otimes_\cO E$ have the same dimension.	
	\end{itemize}	
	For the first point, we note that there is an isomorphism 	
	\[ \mrm \cong (S_\emptyset \otimes_{\cA_0} \cB_0)_{P_0}. \]	
	Since $S_\emptyset \otimes_\cO E$ is a semisimple $\bT_{\emptyset} \otimes_\cO E$-module and the map $\bT_{\emptyset} \to \bT_{\emptyset} \otimes_{\cA_0} \cB_0$ is \'etale at $P_0$, $\mrm$ is indeed semisimple. For the second, we simply apply Lemma \ref{lem_adjoint_Selmer}.	
\end{proof}
\bibliographystyle{amsalpha}
\bibliography{CMpatching}

\end{document}